\documentclass[leqno,11pt]{amsart}

\usepackage{amssymb, amsmath}

\def\refer#1{~\ref{#1}}
\def\refeq#1{~(\ref{#1})}
\def\ccite#1{~\cite{#1}}

\def\longformule#1#2{
\displaylines{ \qquad{#1} \hfill\cr \hfill {#2} \qquad\cr } }
\def\inte#1{
\displaystyle\mathop{#1\kern0pt}^\circ }

\def\uapp{u_{\varepsilon}^{app}}

\let\al=\alpha

\let\d=\delta
\let\e=\varepsilon
\let\eps=\varepsilon

\let\s=\sigma
\let\f=\phi

\let\wh=\widehat

\def\cG{{\mathcal G}}

\def\dive{\mathop{\rm div}\nolimits}

\def\virgp{\raise 2pt\hbox{,}}
\def\cdotpv{\raise 2pt\hbox{;}}

\def\eqdefa{\buildrel\hbox{{\rm \footnotesize def}}\over =}
\def\Id{\mathop{\rm Id}\nolimits}
\def\C{\mathop{\mathbb C\kern 0pt}\nolimits}
\def\DD{\mathop{\mathbbf D\kern 0pt}\nolimits}
\def\K{\mathop{\mathbb K\kern 0pt}\nolimits}
\def\N{\mathop{\mathbb N\kern 0pt}\nolimits}
\def\Q{\mathop{\mathbbf Q\kern 0pt}\nolimits}
\def\R{\mathop{\mathbb R\kern 0pt}\nolimits}
\def\SS{\mathop{\mathbb S\kern 0pt}\nolimits}
\def\ZZ{\mathop{\mathbb Z\kern 0pt}\nolimits}
\def\TT{\mathop{\mathbb T\kern 0pt}\nolimits}
\def\P{\mathop{\mathbb P\kern 0pt}\nolimits}

 \def\R{{\mathop{\mathbb R\kern 0pt}\nolimits}}

\newcommand{\beq}{\begin{equation}}
\newcommand{\eeq}{\end{equation}}
\newcommand{\ben}{\begin{eqnarray}}
\newcommand{\een}{\end{eqnarray}}
\newcommand{\beno}{\begin{eqnarray*}}
\newcommand{\eeno}{\end{eqnarray*}}

\newtheorem{thm}{Theorem}
\newtheorem{lemma}{Lemma}[section]
\newtheorem{rmk}{Remark}[section]
\newtheorem{corol}{Corollary}[section]
\newtheorem{prop}{Proposition}[section]

\begin{document}

\title[The role of spectral anisotropy in   the Navier-Stokes equations]{The role of spectral anisotropy in the resolution of the  three-dimensional Navier-Stokes equations}

\author[J.-Y. Chemin]{Jean-Yves  Chemin}
\address[J.-Y. Chemin]%
{ Laboratoire J.-L. Lions UMR 7598\\ Universit{\'e} Paris VI\\
175, rue du Chevaleret\\ 75013 Paris\\FRANCE }
\email{chemin@ann.jussieu.fr }
\author[I. Gallagher]{Isabelle Gallagher}
\address[I. Gallagher]%
{ Institut de Math{\'e}matiques de Jussieu UMR 7586\\ Universit{\'e} Paris VII\\
175, rue du Chevaleret\\ 75013 Paris\\FRANCE }
\email{Isabelle.Gallagher@math.jussieu.fr}
\author[C. Mullaert]{Chlo\'e Mullaert}
\address[C. Mullaert]%
{ Laboratoire J.-L. Lions UMR 7598\\ Universit{\'e} Paris VI\\
175, rue du Chevaleret\\ 75013 Paris\\FRANCE }
\email{cmullaert@ann.jussieu.fr }


\begin{abstract}
We present   different classes of initial data to the three-dimensional, incompressible Navier-Stokes equations, which generate a global in time, unique solution though they may be arbitrarily large in the end-point function space in which a fixed-point argument may be used to solve the equation locally in time. The main feature of these initial data is an anisotropic distribution of their frequencies. One of those classes is taken from~\cite{cg3}-\cite{cgz}, and another one is new.
 \end{abstract}

\keywords {Navier-Stokes equations, global wellposedness, anisotropy.}

\maketitle


\setcounter{equation}{0}
\section{Introduction}
In this article, we are interested in the construction of global smooth solutions which cannot be obtained in the framework of small data. Let us recall what the incompressible Navier-Stokes (with constant density) is:
\[
{(NS)}\ \left\{
\begin{array}{l}
\partial_{t} u  +u\cdot\nabla u-\Delta u=-\nabla p \quad \mbox{in} \quad  \R^+ \times \R^3\\
\dive u =0\\
u_{|t=0}=u_{0}.
\end{array}
\right.
\]
In all this paper~$x= (x_h, x_3)=(x_1,x_2,x_3)$ will denote  a generic point of~$\R^3$ and we shall write~$u=(u^h,u^3)=(u^1,u^2,u^3)$ for a vector field on~$\R^3=\R^2_h\times\R_v$. We also define the    horizontal differentiation operators~$\nabla^h \eqdefa (\partial_1, \partial_2)$ and~$\mbox{div}_h \eqdefa \nabla^h  \, \cdot  $, as well as~$\Delta_h \eqdefa\partial_1^2 + \partial_2^2$.

First, let us recall the history of global existence results for small data.  In his seminal work\ccite{leray}, J. Leray  proved in~1934 that if~$\|u_0\|_{L^2}\|\nabla u_0\|_{L^2}$ is small enough, then there exists a global regular solution of~$(NS)$. Then in\ccite{fk}, H. Fujita and T. Kato proved in~1964 that if 
$$
\|u_0\|_{\dot H^{\frac 12 }} \eqdefa \Bigl(\int_{\R^3} |\xi|\, |\wh u_0(\xi)|^2d\xi\Bigr)^{\frac 1 2}
$$
is small enough, then there exists a unique global solution in the space
$$
C_b(\R^+;\dot H^{\frac 1 2}) \cap L^4(\R^+;\dot H^1).
$$
After works of many authors on this question (see in particular\ccite{giga},\ccite{kato},\ccite{weissler},and\ccite{cannonemeyerplanchon}), the optimal norm to express the smallness of the initial data was found  on 2001 by H. Koch and D. Tataru in\ccite{kochtataru}. This is the~$BMO^{-1}$ norm. We are not going to define precisely this norm here. Let us simply notice that this norm is in between two Besov norms which can be easily defined. More precisely we have
$$
\displaylines{
\|u_0\|_{\dot B^{-1}_{\infty,\infty}} \lesssim \| u_0\|_{BM0^{-1}} \lesssim\| u_0\|_{\dot B^{-1}_{\infty,2}}
\quad\hbox{with}\cr
\| u_0\|_{\dot B^{-1}_{\infty,\infty}}  \eqdefa \sup_{t>0} t^{\frac 1 2} \|e^{t\Delta} u_0\|_{L^\infty} 
\quad \hbox{and}\quad
 \| u_0\|_{\dot B^{-1}_{\infty,2}} \eqdefa \|e^{t\Delta} u_0\|_{L^2(\R^+;L^\infty)}.
}  
$$
Fisrt of all, let us mention that~$\dot H^{\frac 1 2}$ is continuously embedded in~$\dot B^{-1}_{\infty,2}$. To have a more precise idea of what these spaces mean, let us observe that the space~${\dot B^{-1}_{\infty,\infty}}$ we shall denote by~$\dot C^{-1}$ from now on, contains all the derivatives of order~$1$  of 
bounded functions. Let us give some examples. If we consider a divergence free vector field of the type
$$
u _{\e,0}(x)= \frac 1 \e \cos\Bigl(\frac {x_3} \e\Bigr) \bigl(-\partial_2\phi (x), \partial_1\phi(x), 0\bigr)
$$ 
for some given function~$\phi$ in the Schwartz class of~$\R^3$, then we have
$$
\|u_{\e,0}\|_{\dot B^{-1}_{\infty,2}} \sim \|u_{\e,0}\|_{\dot C^{-1}}\sim 1\quad\hbox{and}\quad
\|u_{\e,0}\|_{\dot H^{\frac 12}} \sim \e^{-\frac 32}.
$$
Another example which will be a great interest for this paper is the case when
$$
u _{\e,0}(x)=  \phi_0(\e x_3) \bigl(-\partial_2\phi (x_h), \partial_1\phi(x_h), 0\bigr).
$$
As claimed  by Proposition 1.1 of\ccite{cg3}, we have, for small enough~$\e$,
\beq
\label{minorC-1}
 \|u_{\e,0}\|_{\dot C^{-1}} \geq \frac 12  \|\f\|_{\dot C^{-1}(\R^2)} \|\f_0\|_{L^\infty(\R)}.
\eeq
In this paper, we are going to consider inital data the regularity of which will be (at least)~$\dot H^{\frac 1 2}$. Our interest is focused on the size of the initial data measured in the~$\dot C^{-1}$ norm.

Let us define~${\mathcal G}$ the set of divergence free
vector fields in~$\dot H^{\frac 1 2}(\R^3)$ generating global smooth solutions to~$(NS)$ and let us recall  some known results about the geometry of~$\cG$.

 First of all, Fujita-Kato' theorem~\cite{fk} can be interpreted as follows: the set~$\cG$ contains a ball of positive radius. Next let us assume that~$\cG$ is not the whole space~$\dot H^{\frac 12 }$ (in other words, we assume that an initial data exists which generates singularities in finite time). Then  there exists a critical radius~$\rho_c$ such that if~$u_0$ is an initial data such that~$\|u_0\|_{\dot H^  {\frac 1 2}}<\rho_c$, then~$u_0$ generates a global regular solution and for any~$\rho>\rho_c$,  there  exists an intial data of~$\dot H^{\frac 1 2}$ norm~$\rho$ which  generates a singularity at finite time. Using the  theory of profiles introduced in the context of Navier-Stokes equations  by the second author (see\ccite{gallagherprofil}), W. Rusin and V. Sverak prove in\ccite{rusinsverack} that the set (where~$ \cG^c$ denotes the complement of~$ \cG$ in~$ \dot H^{\frac 1 2}$)
 $$
 \cG^c\cap \bigl\{u_0 \in \dot H^{\frac 1 2}\,/\ \|u_0 \|_{ \dot H^{\frac 1 2}}=\rho_c\bigr\} 
 $$
 is non empty and compact up to dilations and translations.

In collaboration with P. Zhang, the first two authors prove in\ccite{cgz} that  any point~$u_0$ of~$\cG$, 
is at the center of an interval~$I$  included in~$\cG$ and such that the length of~$I$ measured in the~$\dot C^{-1}$ norm is arbitrary large. In other words for any~$u_0$ in~$\cG$, there exist arbitrary large (in~$\dot C^{-1}$) perturbations of this initial data that generate global solutions. As we shall see, the perturbations  are strongly anisotropic.  

Our aim is to give a new point of view about  the important role played  by  anisotropy in the resolution of the Cauchy problem for~$(NS)$.

 The first result we shall present
shows that as soon as enough anisotropy is present in the initial data (where the degree of anisotropy is given by the norm of the data only), then it generates a global unique solution. A similar result can be found in~\cite[Theorem 1]{bahourigallagher}.
 
\begin{thm}
\label{anisodebile}
{\sl A constant~$c_0$ exists which satisfies the following.  If~$(u_{\e,0})_{\e>0}$  is a family of divergence free vector field  in~$ \dot H^\frac12 $ such that~$\|u_{\e,0}\|_{  \dot H^\frac12 } \leq \rho$ and satisfying
\begin{equation}\label{hypsuppdebile}
\forall \xi \in \mbox{\rm Supp\,} \, \widehat u_{\e,0} \, , \quad \mbox{either} \quad |\xi_h| \leq \eps |\xi_3| \quad  \mbox{or} \quad
 |\xi_3| \leq \eps |\xi_h| \, , 
 \end{equation}
then, if~$\e^4\|u_{\e,0}\|_{  \dot H^{\frac12 }}$ is less than~$c_0$,~$u_{\e,0}$ belongs to~$\cG$.
}
\end{thm}
Let us remark that this result has  little to do with the precise structure of the equations: as  will appear clearly in its   proof in Section\refer{prelims}, it
can actually easily  be recast  as a small data theorem, the smallness being measured in anisotropic Sobolev spaces.
It is therefore of a different nature than the next Theorems~{\rm\ref{slowvar}} and~{\rm\ref{spectral}}, whose proofs on the contrary rely heavily on the structure of the nonlinearity (more precisely on the fact that the two-dimensional equations are globally well-posed).

The next theorem shows that as soon as the initial data has slow variations in one direction, then it generates a global solution, which, roughly speaking, corresponds to the case when the support in Fourier space of the initial data lies in the region where~$ |\xi_3| \leq \eps |\xi_h|$.
Furthermore, one can add to any initial data in~$\cG$ any such slowly varying data, and the superposition still generates a global solution (provided the variation is slow enough and   the profile vanishes at zero).
\begin{thm}[\cite{cg3},\cite{cgz}]\label{slowvar}
 {\sl Let~$v_0^h=(v_0^1,v_0^2)$ be a  horizontal, smooth
 divergence free  vector field on~$\R^3$ (i.e.~$v_0^h$ is  in~$L^2(\R^3)$ as well as all its derivatives), 
 belonging, as well as all its derivatives, to~$L^2(\R_{x_3}; \dot H^{-1}(\R^2))$; let $w_0$ be 
 a smooth divergence free vector field on~$\R^3$. Then, there exists a positive~$\e_0$ depending on norms of~$v_0^h$ and~$w_0$ such that,  if~$\e\leq \e_0$, then the following initial data belongs to~$\cG$ :
$$
v_{\e,0} (x) \eqdefa   ( v_0^h+\e w^h_0,w_0^3 )(x_1,x_2,\e x_3) \, .
$$
If moreover~$v^h_0(x_1,x_2,0)= w_0^3(x_1,x_2,0) =0$ for all $(x_1,x_2)\in\R^2$, and if~$u_0$ belongs to~$\cG$, then there exists a positive number~$\e'_0 $ depending on~$u_0$ and
on norms of~$v_0^h$ and~$w_0$ such that if~$\e\leq \e'_0$,
the following initial data belongs to~$\cG$ :
$$
u_{\e,0}  \eqdefa u_0 + v_{\e,0}   \, .
$$
}
\end{thm}

One can assume  that~$v_0^h$ and~$w_0^3$ have frequency supports in a given ring of~$\R^3$, so that~{\rm(\ref{hypsuppdebile})} holds. Nevertheless Theorem~{\rm\ref{anisodebile}}  not apply since~$v_{\e,0} $ is of the order of~$\eps^{-\frac12} $ in~$\dot H^\frac12$. Actually the proof of Theorem~{\rm\ref{slowvar}} is deeper than that of Theorem~{\rm\ref{anisodebile}}, as it uses the structure of the quadratic term in~{\rm(NS)}.
 The proof of Theorem~{\rm\ref{slowvar}} may be found in~\cite{cg3} and~\cite{cgz}, we shall not give it here. 
Note that  Inequality\refeq{minorC-1} implies that~$v_{\e,0}$ may be chosen arbitrarily large in~$\dot C^{-1}$.

One formal way to translate the above result is that the vertical frequencies of the initial data~$v_{\eps,0}$ are actually very small, compared with the horizontal frequencies. The following theorem gives a   statement in terms of frequency sizes, in the spirit of Theorem~{\rm\ref{anisodebile}}. However as already pointed out,
 Theorem~{\rm\ref{anisodebile}} again  does not apply because the initial data is too large in~$\dot H^\frac12$. 
Notice also that the assumption made in  the statement of Theorem~\ref{slowvar} that the profile should vanish at~$x_3 = 0$ is replaced here by a smallness assumption in~$L^2(\R^2)$.
\begin{thm}
\label{spectral}
{\sl Let~$(v_{\e,0})_\e$ be a family of smooth divergence free vector field, uniformly bounded in the space~$L^\infty(\R;\dot H^s(\R^2))$ for all~$ s \geq -1$,   such that~$
(\sqrt\eps \, v_{\e,0} )_\e$ is uniformly bounded in the space~$ L^2(\R_{x_3};\dot H^s(\R^2)) $ for~$ s \geq -1$, 
and satisfying
$$
\forall \eps \in ]0,1[ \, , \quad \forall \xi \in \mbox{\rm Supp\,} \, \widehat v_{\e,0} \, , \quad   
 |\xi_3| \leq  \eps |\xi_h|  \, .
$$
Then there exists a positive number~$\e_0 $ such that for all~$\eps \leq \eps_0$, the data~$v_{\e,0}$  belongs to~$\cG$.

Moreover if~$u_0$ belongs to~$\cG$, then there are positive constants~$c_0$ and~$\e_0'$ such that if
$$
\|v_{\e,0}(\cdot,0) \|_{L^2(\R_h^2)} \leq c_0
$$
then for all~$\e\leq \e'_0$,
 the following initial data belongs to~$\cG$ :
$$
u_{\e,0}  \eqdefa u_0 + v_{\e,0} \,  .
$$
}
\end{thm}

Let us remark that  as in~\cite{cg3}, the data~$v_{\e,0}$ may be arbitrarily large in~$\dot C^{-1}$. Note that Theorems~{\rm\ref{slowvar}} and~{\rm\ref{spectral}}, though of similar type, are not   comparable (unless one imposes the spectrum of the initial profiles in Theorem~{\rm\ref{slowvar}} to be included in a ring of~$\R^3$, in which case the result follows from Theorem~{\rm\ref{spectral}}). 

The paper is organized as follows.  In the second section, we introduce anisotropic Sobolev spaces and  as a warm up, we prove Theorem\refer{anisodebile}.

The rest of the paper is devoted to the proof of Theorem\refer{spectral}.  In the thid section, we define a (global) approximated solution and prove estimates on this approximated solutions  and prove Theorem\refer{spectral}. 

The last section is devoted to the proof of a propagation result for a linear transport diffusion equation  we admit in the preceeding section. Let us point out that we make the choice not to use the technology anisotropic paradifferential calculus and to present an elementary proof.

\section{Preliminaries: notation and anisotropic function spaces}
\label{prelims}
In this section we recall the definition of the various function spaces we shall be using in this paper, namely     anisotropic Lebesgue and Sobolev spaces.

We denote by~$L^p_hL^q_v$ (resp.~$L^q_v(L^p_h)$) the space~$L^p(\R_h^2;L^q(\R_v))$ (resp.~$L^q(\R_v;L^p(\R_h^2))$ equipped with the norm
$$
\|f\|_{L^p_hL^q_v} \eqdefa\bigg( \int_{\R_h^2}  \Big( \int_{\R_v} | f(x_h,x_3)|^q \, dx_3\Big)^\frac pq \, dx_h\bigg)^\frac1p
$$
and similarly~$\dot H^{s,\s}$ is the space~$\dot H^s(\R^2;\dot H^\sigma(\R))$ with
$$
\|f\|_{\dot H^{s,\s}} \eqdefa \bigg(\int_{\R^3}  |\xi_h|^{2s}  |\xi_3|^{2\sigma} |\widehat f(\xi_h,\xi_3)|^2 \, d\xi_h d\xi_3\bigg)^\frac12 
$$
where~$\widehat f = {\mathcal F}f$ is the Fourier transform of~$f$. Note that~$\dot H^{s,\s}$
  is a Hilbert space as soon as~$s<1$ and~$\sigma<1/2$. We define also
  $$
  \|f\|_{\dot H^{s_1,s_2,s_3}}\eqdefa \bigg(\int_{\R^3}  |\xi_1|^{2s_1}  |\xi_2|^{2s_2}  |\xi_3|^{2 s_3} |\widehat f(\xi_1,\xi_2,\xi_3)|^2 \, d\xi_1   d\xi_2 d\xi_3\bigg)^\frac12 .
  $$
This is a Hilbert space if all~$s_j$ are less than~$1/2$.   Finally we shall often use the spaces~$L^p_v\dot H^s_h = L^p(\R_v;\dot H^s(\R^2_h))$. Let us notice that~$L^2_v\dot H^s_h=\dot H^{s,0}$
The following result, proved by D. Iftimie in\ccite{iftimie} is the basis of the proof of Theorem\refer{anisodebile}.
\begin{thm}
{\sl There is a constant~$\e_0$ such that the following result holds.
Let~$(s_i)_{1 \leq i \leq 3}$  be such that~$s_1+s_2+s_3 = 1/2$ and~$-1/2<s_i<1/2$. Then any divergence free vector field of norm smaller than~$\e_0$ in~$\dot H^{s_1,s_2,s_3}$ generates
a global smooth solution to~$(NS)$.
}
\end{thm}
This theorem implies obviously the following corollary, since~$\dot H^{s,\frac12-s}$ is continuously embedded in~$\dot H^{\frac s 2,\frac s 2,\frac12-s}$ as soon as~$0<s<1/2$. More precisely, we have that the space~$\dot H^{s,\frac 1 2 -s}$ is the space~$\dot H^{s ,0,\frac12-s}\cap \dot H^{0,s,\frac12-s}$.
\begin{corol}\label{dragos}
{\sl There is a constant~$\e_0$ such that the following result holds.
Let~$s$ be given in~$]0,1/2[$ . Then any divergence free vector field of norm smaller than~$\e_0$ in~$\dot H^{s,\frac12-s}$ generates
a global smooth solution to~{\rm(NS)}.
}
\end{corol}
\begin{proof}[Proof of Theorem~\ref{anisodebile}]
 Let us decompose~$u_0$ into two parts, namely we write~$u_0 = v_0+w_0$, with
 $$
 v_0\eqdefa {\mathcal F}^{-1} \big ({\mathbf 1}_{|\xi_h  | \leq  \eps |\xi_3|}\widehat u_0 (\xi) \big) \quad \mbox{and} \quad w_0\eqdefa {\mathcal F}^{-1} \big({\mathbf 1}_{|\xi_3  | \leq \eps |\xi_h|}\widehat u_0 (\xi) \big)\, .
 $$ 
 Let~$0< s < 1/2$ be given. On the one hand we have
 $$
 \|v_0\|_{\dot H^{s,\frac 1 2-s}}^2 = \int_{|\xi_3  | \leq \eps |\xi_h|} |\xi_h|^{2s} |\xi_3|^{1-2s}| \widehat 
  u_0 (\xi)|^2 \, d\xi
  $$
hence since~$ s < 1/2$, 
$$
\begin{aligned}
\|v_0\|_{\dot H^{s,\frac 1 2-s}}^2 & \leq  \eps^{1-2s}  \int  |\xi_h| \,  | \widehat 
  u_0 (\xi)|^2 \, d\xi\\
  & \leq \eps^{1-2s}  \|u_0\|_{\dot H^\frac12}^2 \, .
\end{aligned}
$$
Identical computations give, since~$s>0$,
$$
\begin{aligned}
\|w_0\|_{\dot H^{s,\frac 1 2-s}}^2 & =    \int_{|\xi_h | \leq \eps |\xi_3|}  |\xi_h|^{2s} |\xi_3|^{1-2s} | \widehat 
  u_0 (\xi)|^2 \, d\xi\\
 & \leq   \eps ^{2s}    \int   |\xi_3| \, | \widehat 
  u_0 (\xi)|^2 \, d\xi\\
   & \leq  \eps^{2s}  \|u_0\|_{\dot H^\frac12}^2 \,.
\end{aligned}
$$
 To conclude we can 
choose~$s = 1/4$, which gives
$$
\|u_0\|_{\dot H^{\frac 1 4,\frac 1 4}} \leq \eps^{\frac14} \|u_0\|_{\dot H^\frac12}.
$$
Then, the result follows by the wellposedness of~$(NS)$ in~$\dot H^{\frac14,\frac14}$ given by Corollary~\ref{dragos}.  
\end{proof}
\begin{rmk}
{\rm   The proof of Theorem~{\rm\ref{anisodebile}} does not use the special structure of the nonlinear term in~$(NS)$ as it reduces to checking that the initial data is small in an adequate scale-invariant space. 
}
  \end{rmk}

 \section{Proof of Theorem~\ref{spectral}}\label{spectralproof}
In this section we shall prove the second part of Theorem~\ref{spectral}: we consider an initial data~$u_0 + v_{\e,0}$ satisfying the assumptions of the theorem and we prove that for~$\e>0$ small enough, it generates a global, unique solution to (NS). 
It will be clear from the proof that in the case when~$u_0 \equiv 0$ (which amounts to the first part of Theorem~\ref{spectral}), the assumption that~$v_{\e,0}(x_h,0)$ is small in~$L^2(\R^2)$ is not necessary. Thus the proof of the whole of Theorem~\ref{spectral} will be obtained.

\subsection{Decomposition of the initial data}
The first step of the proof consists in decomposing the initial data as follows.
\begin{prop}
{\sl   Let~$v_{\e,0}$ be a divergence free vector field satisfying
$$
\forall \eps \in ]0,1[\, , \quad  \forall \xi \in \mbox{\rm Supp} \, \widehat v_{\e,0} \, , \quad   
 |\xi_3| \leq  \eps |\xi_h| \, .
$$
Then there exist two divergence free vector fields~$\big(\overline v^h_{\eps,0},0\big)$ and~$w_{\eps,0}$ the spectrum of which is included in that of~$v_{\e,0}$, and such that
$$
v_{\eps,0}= (\overline v^h_{\eps,0},0) + w_{\eps,0} 
\quad\hbox{with} \quad\big|\widehat w_{\eps,0}^h \big | \leq  \eps \big|\widehat w_{\eps,0}^3 \big | \, .
$$
}
\end{prop}
\begin{proof}
Let~$\P_h \eqdefa \Id - \nabla_h \Delta_h^{-1} \mbox{div}_h$ be the Leray projector onto horizontal divergence free vector fields and    define
\begin{equation}\label{defoverlinevandw}
\overline v^h_{\eps,0} \eqdefa\P_h v^h_{\eps,0} \quad \mbox{and} \quad  w_{\eps,0} \eqdefa v_{\eps,0} - (\overline v^h_{\eps,0},0)\, .
\end{equation}
The estimate on~$ w_{\eps,0} $ simply comes from the fact that obviously
$$
\widehat w^h_{\eps,0} (\xi) = \frac{\xi_h \cdot \widehat  v^h_{\eps,0}}{|\xi_h|^2} \, \xi_h \, , 
$$
and therefore since~$v_{\eps,0}$ is divergence free and using the spectral assumption we find
$$
| \widehat w^h_{\eps,0} (\xi) | \leq \frac{| \xi_h \cdot  \widehat v^h_{\eps,0}|}{|\xi_h|} =  \frac{|\xi_3 \widehat v^3_{\eps,0}|}{|\xi_h|}   \leq    \eps  | \widehat v^3_{\eps,0} | =   \eps  | \widehat w^3_{\eps,0} (\xi) | \, .
$$
That proves the proposition.
\end{proof}

\subsection{Construction of an approximate solution  and end of the proof of   Theorem~\ref{spectral}} 
The construction of the approximate solution follows closely the ideas of~\cite{cg3}-\cite{cgz}. We write indeed
$$
v_{\e}^{app} \eqdefa(\overline v_\eps^h,0) + w_\eps \quad \mbox{and} \quad  u_{\e}^{app} \eqdefa u +v_{\e}^{app} \, , 
$$
where~$u$ is the global unique solution associated with~$u_0$ and~$\overline v_\eps^h$ solves the two dimensional Navier-Stokes equations for each given~$x_3$:
\[
{\rm(NS2D)}_{x_3}\ \left\{
\begin{array}{l}
\partial_{t} \overline v_\eps^h  +\overline v_\eps^h\cdot\nabla^h \overline v_\eps^h-\Delta_h \overline v_\eps^h=-\nabla^h \overline p_\eps \quad \mbox{in} \quad \R^+ \times \R^2\\
\dive_h \overline v_\eps^h =0\\
\overline v^h_{\eps|t=0}=\overline v^h_{\eps,0}(\cdot,x_3) \, ,
\end{array}
\right.
\]
while~$w_\eps$ solves the linear transport-diffusion type equation
\[
{\rm(T)}\ \left\{
\begin{array}{l}
\partial_{t} w_\eps +\overline v_\eps^h\cdot\nabla^h w_\eps-\Delta w_\eps =-\nabla q_\eps \quad \mbox{in} \quad \R^+ \times \R^3\\
\dive w_\eps  =0\\
  w_{\eps|t=0}=w_{\eps,0} \, .
\end{array}
\right.
\]
Those vector fields satisfy the following bounds (see Paragraph~\ref{estimates} for a proof).
\begin{lemma}\label{bounds}
{\sl Under the assumptions of Theorem~{\rm\ref{spectral}}, the family~$  u_{\e}^{app}$ is uniformly bounded in~$L^2(\R^+;L^\infty(\R^3))$, and~$\nabla  u_{\e}^{app}$ is uniformly bounded in~$L^2(\R^+;L^\infty_vL^2_h)$.
}
\end{lemma}
Now   define~$u_\eps$ the solution associated with the initial data~$u_0 + v_{\eps,0}$, which a priori has a finite life span, depending on~$\eps$. Consider
$$
R_\eps \eqdefa u_\eps -  u_{\e}^{app} \, ,
$$
which satisfies the following property (see Paragraph~\ref{remainder} for a proof).
\begin{lemma}\label{lemmaquasi2D2}
 {\sl For any positive~$\delta$ there exists~$\eps(\d)$ and~$c(\d)$ such that if
 $$\eps\leq \eps(\d)\quad\hbox{and if}\quad \|v_{\eps,0}(\cdot,0)\|_{L^2_h} \leq c(\d),
 $$
  then the vector field~$R^{\e} \eqdefa u_\eps -  u_{\e}^{app}$ satisfies the equation
$$
{\rm(E_\e)} \left\{
\begin{array}{l}
\partial_t R_\eps+ R_\eps \cdot \nabla R_\eps -\Delta R_\eps+\uapp \cdot \nabla R_\e+R_\e \cdot \nabla \uapp =
 F_\e-\nabla \widetilde q_\eps\\
\dive R_\e = 0\\
R_{\eps|t=0} = 0
\end{array}
\right.
$$ 
with $\|F_\e\|_{L^2(\R^+;\dot H^{-\frac 1 2}(\R^3))} \leq \delta$.
}
\end{lemma}
Assuming those two lemmas to be true, the end of the  proof of Theorem~\ref{spectral} follows very easily using the method given  in~\cite[Section 2]{cg3}: an energy estimate in~$\dot H^\frac12(\R^3)$ on~$({\rm E}_\eps)$, using the fact that the forcing term is as small as needed and that the initial data is zero, gives that~$R_\eps$ is unique, and uniformly bounded   in~$L^\infty(\R^+;\dot H^\frac12) \cap L^2(\R^+;\dot H^\frac32)$. Since the approximate solution  is also unique and globally defined,  Theorem~\ref{spectral} is proved. \qed

\subsection{Proof of the estimates on the approximate solution (Lemma~\ref{bounds})}\label{estimates}
As noted in~\cite[Appendix B]{cgz}, the global solution~$u $ associated with~$u_0 \in \dot H^\frac12$ belongs to~$L^2(\R^+;L^\infty(\R^3))$, and~$\nabla  u $ belongs to~$L^2(\R^+;L^\infty_vL^2_h)$.
So we just need to study~$v_{\eps}^{app}$, which we shall do in two steps: first~$\overline v_\eps^h $, then~$w_\eps$.

\subsubsection{Estimates on $\overline v_\eps^h$}
Due to the spectral assumption on~$\overline v^h_{\eps,0}$, it is easy to see that
$$
\begin{aligned}
\forall \alpha = (\alpha_h,\alpha_3) \in \N^2 \times \N \, , \quad \eps^{\frac12-\alpha_3} \partial^\alpha \overline v^h_{\eps,0} \quad \mbox{is uniformly bounded in} \,  \, L^2_v \dot H^s_h \, , \\
\mbox{and} \quad  \eps^{ -\alpha_3} \partial^\alpha \overline v^h_{\eps,0} \quad \mbox{is uniformly bounded in} \,  \,  L^\infty_v \dot H^s_h \, .
\end{aligned}
$$
Indeed the definition of~$\overline v^h_{\eps,0}$ given in~(\ref{defoverlinevandw}), and   the  spectral assumption  as well as the a priori bounds on~$  v_{\eps,0}$, give directly the first result. To     prove  the second result one uses first the Gagliardo-Nirenberg inequality:
$$
\|\partial^\alpha \overline v^h_{\eps,0} \|_{L^\infty_v \dot H^s_h }^2 \leq \|\partial^\alpha \overline v^h_{\eps,0} \|_{L^2_v \dot H^s_h } \|\partial_3\partial^\alpha \overline v^h_{\eps,0} \|_{L^2_v \dot H^s_h } \, ,
$$
and then the same arguments.
The proof of~\cite[Lemma 3.1 and Corollary 3.1]{cg3} enables us to infer from those bounds the following result.
\begin{prop}
\label{estimatesoverlinev}
{\sl Under the assumptions of Theorem~{\rm\ref{spectral}},
 for all real numbers~$s >-1$ and all~$ \alpha = (\alpha_h,\alpha_3) \in \N^2 \times \N $ 
 there is a constant~$\overline C$ such that the   vector field~$\overline v_\eps^h$ satisfies the following bounds:
$$
\begin{aligned}
\|\partial^\alpha \overline v^h_{\eps}(t) \|_{L^\infty_v \dot H^s_h }^2& + \sup_{x_3 \in \R} \int_0^t 
\|\partial^\alpha \nabla^h \overline v^h_{\eps}(t') \|_{ \dot H^s_h }^2  \, dt' \\
&\qquad + \eps \Big( \|\partial^\alpha \nabla^h \overline v^h_{\eps}(t) \|_{L^2_v \dot H^s_h }^2 + \int_0^t  \|\partial^\alpha \overline v^h_{\eps}(t') \|^2_{ L^2(\R^3)} \, dt' \Big)
 \leq \overline C \,   \eps^{2 \alpha_3} \, .
\end{aligned}
$$
}
\end{prop}

\subsubsection{Estimates on $w_\eps$}
The definition of~${w}_{\varepsilon,0}$ given in~(\ref{defoverlinevandw}), along with the  spectral assumption on $\left(v_{\varepsilon,0}\right)_{\varepsilon>0}$ lead to
$$ 
\forall\varepsilon\in]0,1[\,,\quad\forall\xi\in\mbox{Supp}\,\widehat{w}_{\varepsilon,0}\,,\quad|\xi_{3}|\leq\varepsilon|\xi_{h}|\quad\mathrm{and}\quad\bigl|\widehat{w}_{\varepsilon,0}^{h}\left(\xi\right)\bigr|\leq  \varepsilon\bigl|\widehat{w}_{\varepsilon,0}^{3}\left(\xi\right)\bigr|\,.
$$
The proof of the following result is technical and postponed to section\refer{appendix}.
\begin{prop}
\label{estimatesw}
{\sl Under the assumptions of Theorem~{\rm\ref{spectral}}, $w_\eps^3$ and~$\eps^{-1} w_\eps^h$ are uniformly bounded in the space~$L^{\infty} (\mathbb{R}^{+};\, L_{v}^{\infty}L_{h}^{2} ) \cap L^2 (\mathbb{R}^{+};\, L_{v}^{\infty}\dot H_{h}^{s} )$ for all~$s\geq0$. Moreover
 $\eps^{\frac12-\alpha_3} \partial^\alpha w_{\varepsilon}$ is uniformly bounded in~$L^{\infty} (\mathbb{R}^{+};\, L_{v}^{2}\dot{H}_{h}^{s} ) \cap L^{2} (\mathbb{R}^{+};\, L_{v}^{2}\dot{H}_{h}^{s} )$ for all $s\geq0$ and  all $\alpha = (\alpha_h,\alpha_3) \in \N^2 \times \N $.
 }
\end{prop} 
%
The Gagliardo-Nirenberg inequality and   Sobolev embeddings lead to   Lemma~\ref{bounds}.

\subsection{Proof of the estimates on the remainder (Lemma~\ref{lemmaquasi2D2})}\label{remainder}
 Substracting the equation on~$u_\e^{app}$ from the equation on~$u$ one finds directly that
 $$
 F_\e = (\partial_3^2 \overline v_\e^h, \partial_3 \overline p_\eps)  + w_\eps \cdot \nabla v_\e^{app} + u \cdot \nabla  v_\e^{app} + v_\e^{app} \cdot \nabla  u \, ,
 $$
 which we decompose into~$F_\e = G_\e + H_\e $ with
 $$
  G_\e \eqdefa (\partial_3^2 \overline v_\e^h, \partial_3 \overline p_\eps)  + w_\eps \cdot \nabla v_\e^{app}  \quad \mbox{and} \quad H_\e \eqdefa  u \cdot \nabla  v_\e^{app} + v_\e^{app} \cdot \nabla  u \, .
 $$
 Lemma~\ref{lemmaquasi2D2}
 follows from the two following propositions.
 \begin{prop}
 \label{estimateG}
 {\sl There is a positive constant~$C$  such that  for all~$\eps$ in~$ ]0,1[$,
 $$
 \|G_\eps\|_{L^2(\R^+;\dot H^{-\frac 1 2}(\R^3))} \leq C\eps^\frac12 \, . 
 $$
 }
  \end{prop}
  \begin{proof}
  Let us start by splitting~$G_\e$  in three parts: $G_\e= G_\e^1 + G_\e^2+ G_\e^3$ with
  $$
  G_\e^1 \eqdefa (\partial_3^2 \overline v_\e^h, 0) \, , \, \,    G_\e^2 \eqdefa (0,\partial_3 \overline p_\eps) \, , \quad \mbox{and} \quad 
   G_\e^3 \eqdefa w_\eps \cdot \nabla v_\e^{app}    \, .
 $$
 On the one hand we have obviously
 $$
 \|  G_\e^1\|_{L^2(\R^+;\dot H^{-\frac 1 2}(\R^3))}  \leq \|\partial_3   \overline v_\e^h\|_{L^2(\R^+;\dot H^{\frac 1 2}(\R^3))} \, .
 $$
Proposition~\ref{estimatesoverlinev} applied with~$\alpha = (0,1)$, $\alpha = (0,2)$ and $\alpha = (\al_h,1) $ with~$|\al_h|=1$ gives
$$
\int_0^t \|\partial_3\overline v_\e(t',\cdot)\|_{L^2}^2dt' \lesssim \e\quad\hbox{and}\quad
\int_0^t \|\partial_3\nabla \overline v_\e(t',\cdot)\|_{L^2}^2dt' \lesssim \e.
$$
By interpolation, we infer that 
\beq
\label{estimateGdemoeq1}
 \|  G_\e^1\|_{L^2(\R^+;\dot H^{-\frac 1 2}(\R^3))}  \lesssim \eps^\frac12 \, .
\eeq
To estimate~$G_\e^2$ we use the fact that
$$
-\Delta_h  \overline p_\eps = \sum_{j,k=1}^2 \partial_j \partial_k (  \overline v_\e^j   \overline v_\e^k)
$$
and since~$(-\Delta_h)^{-1} \partial_j \partial_k $ is a Fourier multiplier of order 0 for each~$(j,k)$ in~$ \{1,2\}^2$ we get
$$
 \|  G_\e^2\|_{L^2(\R^+;\dot H^{-\frac 1 2}(\R^3))} \lesssim \sum_{j,k=1}^2 \big \|\overline v_\e^j \partial_3  \overline v_\e^k \big\|_{L^2(\R^+;\dot H^{-\frac 1 2}(\R^3))} \, .
$$
As~$ L^2_v \dot H^{-\frac12}_h \hookrightarrow \dot H^{-\frac 1 2}(\R^3)$, we get
$$
 \|  G_\e^2\|_{L^2(\R^+;\dot H^{-\frac 1 2}(\R^3))}  \lesssim \sum_{j,k=1}^2 \big \|\overline v_\e^j \partial_3  \overline v_\e^k \big\|_{L^2(\R^+;L^2_v \dot H^{-\frac12}_h)} \, .
$$  
Using the Sobolev embedding~$ L^\frac43_h \hookrightarrow\dot H^{-\frac12}_h$ and H\"older's inequality gives
$$
\begin{aligned}
 \|  G_\e^2\|_{L^2(\R^+;\dot H^{-\frac 1 2}(\R^3))} & \lesssim  \sum_{j,k=1}^2 \big \|\overline v_\e^j \partial_3  \overline v_\e^k \big\|_{L^2(\R^+;L^2_v L^\frac43_h)} \\
 & \leq C  \| \overline v_\e^h\|_{L^\infty(\R^+;L^\infty_v L^{\frac83}_h)} \| \partial_3\overline v_\e^h\|_{L^\infty(\R^+;L^\infty_v L^{\frac83}_h)}
 \end{aligned}
$$
so the Sobolev embedding~$   \dot H^{\frac14}_h\hookrightarrow L^\frac83_h$ gives finally
$$
\begin{aligned}
 \|  G_\e^2\|_{L^2(\R^+;\dot H^{-\frac 1 2}(\R^3))} & \lesssim C \| \overline v_\e^h\|_{L^\infty(\R^+;L^\infty_v \dot H^{\frac14}_h)}
  \|\partial_3 \overline v_\e^h\|_{L^2(\R^+;L^2_v \dot H^{\frac14}_h)} \, .
\end{aligned}
$$
The result follows again from Proposition~\ref{estimatesoverlinev}: choosing~$s = 1/4$ and~$\alpha = 0$ we get that~$ \overline v_\e^h$ is uniformly bounded in~$L^\infty(\R^+;L^\infty_v \dot H^{\frac14}_h)$, while~$s=-3/4$ and~$\alpha = (\alpha_h,1)$ with~$|\alpha_h|=1$ gives
$$
\|\partial_3 \overline v_\e^h\|_{L^2(\R^+;L^2_v \dot H^{\frac14}_h)} \lesssim \eps^\frac12 \, .
$$
We infer finally that
\beq
\label{estimateGdemoeq2}
 \|  G_\e^2\|_{L^2(\R^+;\dot H^{-\frac 1 2}(\R^3))}  \lesssim \eps^\frac12 \, .
\eeq
To end the proof of the proposition let us estimate~$G_\eps^3$. We simply use two-dimensional product laws, which gives
$$
\begin{aligned}
 \|  G_\e^3\|_{L^2(\R^+;\dot H^{-\frac 1 2}(\R^3))}  & = \|w_\eps \cdot \nabla  v_\eps^{app}\|_{L^2(\R^+;\dot H^{-\frac 1 2}(\R^3))} \\
 & \lesssim \|w_\eps^h\|_{L^\infty(\R^+;L^\infty_v \dot H^\frac14_h)}  \|\nabla^h  v_\eps^{app}\|_{L^\infty(\R^+;L^2_v \dot H^\frac14_h)} \\
 & \qquad \qquad +  \|w_\eps^3\|_{L^\infty(\R^+;L^\infty_v \dot H^\frac14_h)}  \| \partial_3  v_\eps^{app}\|_{L^\infty(\R^+;L^2_v \dot H^\frac14_h)} 
 \lesssim \eps^\frac12 \, ,
  \end{aligned}
$$
due to Propositions~\ref{estimatesoverlinev} and~\ref{estimatesw}.
Together with Inequalities\refeq{estimateGdemoeq1} \and\refeq{estimateGdemoeq2} that proves Proposition~\ref{estimateG}.
\end{proof}
\begin{prop}\label{estimateH}
{\sl Let~$\delta>0$ be given. There are positive constants~$\eps(\d)$ and~$c(\d)$ such that if~$\eps\leq \eps(\d)$ and if~$\|v_{\eps,0}(\cdot,0)\|_{L^2_h} \leq c(\d)$, then 
 $$
  \|H_\eps\|_{L^2(\R^+;\dot H^{-\frac 1 2}(\R^3))} \leq \delta  .
 $$
 }
  \end{prop}
  \begin{proof}
First, we approximate $H_\eps$, and then we estimate this approximation. 

 Using~\cite[Theorem 2.1]{gallip} we get
$$
\underset{t\rightarrow\infty}{\mbox{lim}}\Vert  u\left(t,.\right)\Vert   _{\dot{H}^{\frac12}\left(\mathbb{R}^{3}\right)}=0
$$
so we can approximate $u$ in $L^{\infty} (\mathbb{R}^{+},\dot{H}^{\frac12} )$: for all $\eta >0$, there exists an integer $N$, real numbers~$\left(t_{j}\right)_{0\leq j\leq N}$ and smooth,  compactly supported, divergence free functions~$\left(\phi_{j}\right)_{1\leq j\leq N}$ such that 
$$
\widetilde u_{\eta}(t) \eqdefa \sum_{j=1}^N \boldsymbol{1}_{\left[t_{j-1},t_{j}\right]}\left(t\right)\phi_{j}
$$
is uniformly bounded in~$L^{\infty} (\mathbb{R}^{+},\dot{H}^{\frac12} ) \cap L^{2} (\mathbb{R}^{+},\dot{H}^{\frac32} ) $
and satisfies
\begin{equation}\label{smalleta}
\Vert  u-\widetilde u_{\eta}\Vert   _{L^{\infty} (\mathbb{R}^{+},\dot{H}^{\frac12} (\mathbb{R}^{3} ) )}\leq\eta\,.
\end{equation}
We split $H_\eps$ into two contributions 
$$
H_\eps=H_{\varepsilon,\eta} + (\widetilde u_{\eta}-u )\cdot\nabla v_{\varepsilon}^{app}+v_{\varepsilon}^{app}\cdot\nabla (\widetilde u_{\eta}-u )$$ 
with $H_{\varepsilon,\eta} \eqdefa\widetilde u_{\eta}\cdot\nabla v_{\varepsilon}^{app}+v_{\varepsilon}^{app}\cdot\nabla\widetilde u_{\eta} \, .$

As  $v_{\varepsilon}^{app}$ and~$\widetilde u_{\eta}-u$ are divergence free vector fields, 
$$
H_{\varepsilon} -H_{\varepsilon,\eta} =\mbox{div}\big( (\widetilde u_{\eta}-u )\otimes v_{\varepsilon}^{app}+v_{\varepsilon}^{app}\otimes (\widetilde u_{\eta}-u ) \big)\,.
$$
Thanks to~\cite[Lemma 3.3]{cgz} we get
$$
\Vert  H_{\varepsilon} -H_{\varepsilon,\eta} \Vert   _{\dot{H}^{-\frac12}}\lesssim \Vert \widetilde u_{\eta}-u \Vert _{\dot{H}^{\frac12}}\big( \Vert \nabla^hv_{\varepsilon}^{app} \Vert _{L_{v}^{\infty}L_{h}^{2}}+ \Vert v_{\varepsilon}^{app} \Vert _{L^{\infty}}+ \Vert \partial_{3}v_{\varepsilon}^{app} \Vert _{L_{v}^{2}\dot{H}_{h}^{\frac12}} \big)
$$
and Proposition \ref{estimatesoverlinev} along with~(\ref{smalleta}) lead to 
$$
\Vert   H_{\varepsilon} -H_{\varepsilon,\eta} \Vert   _{L^{2}(\mathbb{R}^{+},\dot{H}^{-\frac12}\left(\mathbb{R}^{3}\right))}\lesssim\eta\,.$$

It remains to estimate $H_{\varepsilon,\eta} =\widetilde u_{\eta}\cdot\nabla v_{\varepsilon}^{app}+v_{\varepsilon}^{app}\cdot\nabla\widetilde u_{\eta}$.
By Propositions~\ref{estimatesoverlinev} and~\ref{estimatesw} we have 
$$
 \Vert \widetilde u_{\eta}^{3}\partial_{3}v_{\varepsilon}^{app} \Vert _{L^{2} (\mathbb{R}^{+},\dot{H}^{-\frac12} (\mathbb{R}^{3} ) )}\begin{array}[t]{l}
\lesssim \Vert \widetilde u_{\eta}^{3} \Vert _{L^{\infty} (\mathbb{R}^{+},\dot{H}^{\frac12} (\mathbb{R}^{3} ) )} \Vert \partial_{3}v_{\varepsilon}^{app} \Vert _{L^{2} (\mathbb{R}^{+},\dot{H}^{\frac12} (\mathbb{R}^{3} ) )}\\
\lesssim \Vert \widetilde u_{\eta}^{3} \Vert _{L^\infty(\mathbb{R}^{+},\dot{H}^{\frac12} (\mathbb{R}^{3} ))} \, \varepsilon^{\frac12}.
\end{array}
$$
 Since $\widetilde u_{\eta}$ is uniformly bounded in~$L^{\infty} (\mathbb{R}^{+},\dot{H}^{\frac12} (\mathbb{R}^{3} ) )$, we infer that
$$ 
\lim_{\varepsilon\rightarrow0}  \Vert \widetilde u_{\eta}^{3}\partial_{3}v_{\varepsilon}^{app} \Vert _{L^{2} (\mathbb{R}^{+},\dot{H}^{-\frac12} (\mathbb{R}^{3} ) )}=0\,.
$$
 Lemma 3.4 of~\cite{cgz} claims that
 $$
\|ab\|_{\dot H^{-\frac 1 2}} \leq C\|a\|_{L^2_v\dot H^{\frac 1 2}_h}
\|b(\cdot ,0)\|_{L^2_h} + C\|x_3 a \|_{L^2}
\|\partial_3b\|_{L^\infty_v\dot H^{\frac 1 2}_h}.
$$
So we  get
$$
  \Vert \widetilde u_{\eta}^{h}\cdot\nabla^hv_{\varepsilon}^{app} \Vert _{\dot{H}^{-\frac12}}\lesssim \Vert \widetilde u_{\eta}^{h} \Vert _{L_{v}^{2}\dot{H}_{h}^{\frac12}} \Vert \nabla^hv_{\varepsilon}^{app}\left(\cdot,0\right) \Vert _{L_{h}^{2}}+ \Vert x_3\widetilde u_{\eta}^{h} \Vert _{L^{2}} \Vert \partial_{3}\nabla^hv_{\varepsilon}^{app} \Vert _{L_{v}^{\infty}\dot{H}_{h}^{\frac12}}
 $$
 and
 $$
 \Vert  v_{\varepsilon}^{app}\cdot\nabla\widetilde u_{\eta}\Vert   _{\dot{H}^{-\frac12}}\lesssim\Vert  \nabla\widetilde u_{\eta}\Vert   _{L_{v}^{2}\dot{H}_{h}^{\frac12}}\Vert  v_{\varepsilon}^{app}\left(\cdot,0\right)\Vert   _{L_{h}^{2}}+\Vert  x_3\nabla\widetilde u_{\eta}\Vert   _{L^{2}}\Vert  \partial_{3}v_{\varepsilon}^{app}\Vert   _{L_{v}^{\infty}\dot{H}_{h}^{\frac12}}\,.
 $$
Propositions~\ref{estimatesoverlinev} and~\ref{estimatesw} lead  to 
$$
\lim_{\varepsilon\rightarrow0} \int_{\mathbb{R}+} \Vert x_3\widetilde u_{\eta}^{h} (t)\Vert _{L^{2}\left(\mathbb{R}^{3}\right)}^{2} \Vert \partial_{3}\nabla^hv_{\varepsilon}^{app}  (t)\Vert _{L_{v}^{\infty}\dot{H}_{h}^{\frac12}\left(\mathbb{R}^{3}\right)}^{2} \, dt =0
$$
and
$$
\lim_{\varepsilon\rightarrow0} \int_{\mathbb{R}+} \Vert x_3\nabla\widetilde u_{\eta}  (t)\Vert _{L^{2}\left(\mathbb{R}^{3}\right)}^{2} \Vert \partial_{3}v_{\varepsilon}^{app}  (t)\Vert _{L_{v}^{\infty}\dot{H}_{h}^{\frac12}\left(\mathbb{R}^{3}\right)}^{2} \, dt =0 \,.
$$
Now we recall that~$\widetilde u_{\eta}$ is uniformly bounded in~$L^{\infty} (\mathbb{R}^{+},\dot{H}^{\frac12} ) \cap L^{2} (\mathbb{R}^{+},\dot{H}^{\frac32} ) $, hence~$\widetilde u_{\eta}$ is uniformly bounded in~$L^{\infty} (\mathbb{R}^{+},L^2_v\dot{H}^{\frac12}_h ) $ and~$\nabla \widetilde u_{\eta}$  is uniformly bounded in~$L^{2} (\mathbb{R}^{+},L^2_v \dot{H}^{\frac12}_h ) $. So in order to to conclude  we   just have to estimate 
$$
 \Vert v_{\varepsilon}^{app} (\cdot,0 ) \Vert _{L^{\infty} (\mathbb{R}^{+},L_{h}^{2} (\mathbb{R}^{2} ) )} +  \Vert \nabla^hv_{\varepsilon}^{app} (\cdot,0 ) \Vert _{L^{2} (\mathbb{R}^{+},L_{h}^{2} (\mathbb{R}^{2} ) )} \, .
$$
 This is done in the following proposition, which concludes the proof of Proposition~\ref{estimateH}.
   \end{proof}

\begin{prop}\label{estimateat0}
{\sl For all $\delta >0$ there are  positive constants $\varepsilon(\delta )$ and $c (\delta )$ such that for all~$0<\varepsilon\leq\varepsilon(\delta )$, if $ \Vert u_{\varepsilon,0} (\cdot, 0 ) \Vert _{L_{h}^{2}}\leq c (\delta )$ then 
$$
 \Vert v_{\varepsilon}^{app} ( \cdot,0) \Vert _{L^{\infty} (\mathbb{R}^{+},L_{h}^{2} (\mathbb{R}^{2} ) )} +  \Vert \nabla^hv_{\varepsilon}^{app} (\cdot, 0 ) \Vert _{L^{2} (\mathbb{R}^{+},L_{h}^{2} (\mathbb{R}^{2} ) )}\leq\delta\,. 
$$
}
 \end{prop}
  \begin{proof}   First, we estimate~$\overline v^h_{\varepsilon}$ and~$w_{\varepsilon}^h$. 
For all $\varepsilon>0$, an energy estimate in $L_{h}^{2}$ gives
\begin{equation}\label{energyx30}
\dfrac{1}{2} \Vert \overline v^h_{\varepsilon} (t,\cdot,0 ) \Vert _{L_{h}^{2}}^{2}+\displaystyle \int_{0}^{t} \Vert \nabla^h\overline v^h_{\varepsilon} (t',\cdot ,0 ) \Vert _{L_{h}^{2}}^{2}\mbox{d}t'=\dfrac{1}{2} \Vert v_{\varepsilon,0} (\cdot, 0 ) \Vert _{L_{h}^{2}}^{2}\,.
\end{equation}
Then, for all  $\delta >0$ there is a constant $c (\delta )$ such that if $ \Vert v_{\varepsilon,0} (\cdot, 0 ) \Vert _{L_{h}^{2}}\leq c (\delta )$ then 
$$
 \Vert \overline v^h_{\varepsilon} ( \cdot,0 ) \Vert _{L^{\infty} (\mathbb{R}^{+},L_{h}^{2} (\mathbb{R}^{2} ) )}
 + \Vert \nabla^h \overline v^h_{\varepsilon} (\cdot, 0 ) \Vert _{L^{2} (\mathbb{R}^{+},L_{h}^{2} (\mathbb{R}^{2} ) )}
 \leq\delta\,.
$$
Moreover, by Proposition~\ref{estimatesw} we have 
$$
 \Vert w_{\varepsilon}^{h} ( \cdot, 0 ) \Vert _{L^{\infty} (\mathbb{R}^{+},L_{h}^{2} )} +  \Vert \nabla^hw_{\varepsilon}^{h} (\cdot, 0 ) \Vert _{L^{2} (\mathbb{R}^{+},L_{h}^{2} )}\lesssim\varepsilon\,.
$$
It remains to estimate~$w_{\varepsilon}^3$.
According to Proposition~\ref{estimatesw}, $w_{\varepsilon}$  and $\nabla^hw_{\varepsilon}$ are uniformly bounded respectively in $L^{\infty}(\mathbb{R}^{+},L_{v}^{\infty}\dot{H}_{h}^{-\frac12})$ and  $L^{2}(\mathbb{R}^{+},L_{v}^{\infty}\dot{H}_{h}^{-\frac12})$, so we shall get the result by proving that for all $\delta >0$ there are  positive constants  $\varepsilon(\delta )$ and $c (\delta )$ such that if $\varepsilon\leq\varepsilon(\delta )$ and~$ \Vert u_{\varepsilon,0} (\cdot,0) \Vert _{L_{h}^{2}}\leq c (\delta )$ then 
$$
 \Vert w_{\varepsilon}^{3} (\cdot,0) \Vert _{L^{\infty} (\mathbb{R}^{+},\dot{H}_{h}^{\frac12} )}+  \Vert \nabla^hw_{\varepsilon}^{3} (\cdot,0) \Vert _{L^{2} (\mathbb{R}^{+},\dot{H}_{h}^{\frac12} )}\leq\delta\,.
$$
Recall that $w_{\varepsilon}^{3}$ satisfies 
$$
\left\{ \begin{array}{l}
\partial_{t}w_{\varepsilon}^{3}+\overline v^h_{\varepsilon}\cdot\nabla^hw_{\varepsilon}^{3}-\Delta_h w_{\varepsilon}^{3}  = \partial_{3}^{2}w_{\varepsilon}^{3}- \partial_{3}q_{\varepsilon}\\[2mm]
w_{\varepsilon|t=0}^{3} =w_{\varepsilon,0}^{3}\,.\end{array}\right.
$$
Define~$
T_{\varepsilon} \eqdefa \partial_{3}^{2}w_{\varepsilon}^{3}-\partial_{3}q_{\varepsilon}\, .
$
An energy estimate in $\dot{H}_{h}^{\frac12}$ gives
\begin{equation}\label{energyweps}
\begin{aligned}
\displaystyle &  \Vert w_{\varepsilon}^{3}\left(t,0\right) \Vert _{\dot{H}_{h}^{\frac12}}^{2} +\int_{0}^{t} \Vert \nabla^hw_{\varepsilon}^{3}\left(t',0\right) \Vert _{\dot{H}_{h}^{\frac12}}^{2}\mbox{d}t'\\[2mm]
&\,  \lesssim    \Vert w_{\varepsilon,0}^{3} (\cdot,0) \Vert _{\dot{H}_{h}^{\frac12}}^{2}+ \Vert  T_{\varepsilon} (\cdot,0) \Vert _{L^{2} (\mathbb{R}^{+},\dot{H}_{h}^{-\frac12} )}^{2}+{\displaystyle \int_{0}^{t} \big | \langle \overline v^h_{\varepsilon}\cdot\nabla^hw_{\varepsilon}^{3},w_{\varepsilon}^{3} \rangle _{\dot{H}_{h}^{\frac12}} \big| (t',0 )\mbox{d}t'\,.}\end{aligned}
\end{equation}
Using~\cite[Lemma 1.1]{chemin10}  we get for each fixed~$x_3$
$$
 \big| \langle  \overline v^h_{\varepsilon}\cdot\nabla^hw_{\varepsilon}^{3},w_{\varepsilon}^{3} \rangle _{\dot{H}_{h}^{\frac12}} (x_3)\big|\lesssim \Vert \nabla^h \overline v^h_{\varepsilon} (x_3)\Vert _{L_{h}^{2}} \Vert \nabla^hw_{\varepsilon}^{3} (x_3)\Vert _{\dot{H}_{h}^{\frac12}} \Vert w_{\varepsilon}^{3} (x_3)\Vert _{\dot H_{h}^{\frac12}}.
$$
In particular, using~(\ref{energyx30}), we get 
$$
\begin{array}[t]{l}
{\displaystyle \int_{0}^{t}\big| \langle  \overline v^h_{\varepsilon}\cdot\nabla^hw_{\varepsilon}^{3},w_{\varepsilon}^{3} \rangle _{\dot{H}_{h}^{\frac12}}\left(t',0\right)\big|\mbox{d}t'}\\[2mm]
\qquad\qquad\quad\lesssim \Vert \nabla^h \overline v^h_{\varepsilon} (\cdot,0) \Vert _{L^{2} (\mathbb{R}^{+},L_{h}^{2} )} \Vert \nabla^hw_{\varepsilon}^{3} (\cdot,0) \Vert _{L^{2} (\mathbb{R}^{+},\dot{H}_{h}^{\frac12} )} \Vert w_{\varepsilon}^{3} (\cdot,0) \Vert _{L^{\infty} (\mathbb{R}^{+},\dot H_{h}^{\frac12} )}\\[2mm]
\qquad\quad\qquad\lesssim \Vert  \overline v^h_{\varepsilon,0} (\cdot,0) \Vert _{L_{h}^{2}} \Vert \nabla^hw_{\varepsilon}^{3} (\cdot,0) \Vert _{L^{2} (\mathbb{R}^{+},\dot{H}_{h}^{\frac12} )} \Vert w_{\varepsilon}^{3} (\cdot,0) \Vert _{L^{\infty} (\mathbb{R}^{+},\dot{H}_{h}^{\frac12} )}\\[2mm]
\end{array}
$$
Then we infer that
$$
\longformule{
  \int_{0}^{t}\big| \langle  \overline v^h_{\varepsilon}\cdot\nabla^hw_{\varepsilon}^{3},w_{\varepsilon}^{3} \rangle _{\dot{H}_{h}^{\frac12}}\left(t',\cdot,0\right)\big|\mbox{d}t' \lesssim \Vert u_{\varepsilon,0} (\cdot,0) \Vert _{L_{h}}
  }
  {{}\times
   \Big( \Vert \nabla^hw_{\varepsilon}^{3} (\cdot, 0 ) \Vert _{L^{2} (\mathbb{R}^{+},\dot{H}_{h}^{\frac12} )}^{2}+ \Vert w_{\varepsilon}^{3} (\cdot, 0 ) \Vert _{L^{\infty} (\mathbb{R}^{+},\dot{H}_{h}^{\frac12} )}^{2} \Big)\,. 
   }
$$
Plugging this inequality into~(\ref{energyweps}) we obtain that   there is a constant $C$ such that
$$
\begin{aligned}
  \Vert w_{\varepsilon}^{3} (\cdot,0) \Vert _{L^{\infty} (\mathbb{R}^{+},\dot{H}_{h}^{\frac12} )}^{2} & + \big(1-C \Vert u_{\varepsilon,0} (\cdot,0) \Vert _{L_{h}^{2}}\big)  \Vert \nabla^hw_{\varepsilon}^{3} (\cdot,0) \Vert _{L^{2} (\mathbb{R}^{+},\dot{H}_{h}^{\frac12} )}^{2} \\ 
& \lesssim \Vert w_{\varepsilon,0}^{3} (\cdot,0) \Vert _{\dot{H}_{h}^{\frac12}}^{2}+  \Vert T_{\varepsilon} (\cdot,0) \Vert _{L^{2} (\mathbb{R}^{+},\dot{H}_{h}^{-\frac12} )}^{2}\\ 
& \lesssim \Vert w_{\varepsilon,0}^{3} (\cdot,0) \Vert _{L_{h}^{2}} \Vert w_{\varepsilon,0}^{3} (\cdot,0) \Vert _{\dot{H}_{h}^{1}}+ \Vert T_{\varepsilon} (\cdot,0) \Vert _{L^{2} (\mathbb{R}^{+},\dot{H}_{h}^{-\frac12} )}^{2}\,.\end{aligned}
$$
As $w_{\varepsilon,0}$ is uniformly bounded in $L_{v}^{\infty}\dot{H}_{h}^{1}$, it remains to prove that 
$$
\lim_{\varepsilon\rightarrow0}  \Vert T_{\varepsilon}  (\cdot,0) \Vert _{L^{2} (\mathbb{R}^{+},\dot{H}_{h}^{-\frac12}  )}=0 \, .
$$
As $\partial_{3}^{2}w_{\varepsilon}^{3}=-\partial_{3}\mbox{div}_{h}w_{\varepsilon}^{h}$, we get
$$
 \Vert \partial_{3}^{2}w_{\varepsilon}^{3} (\cdot,0) \Vert _{L^{2} (\mathbb{R}^{+},\dot{H}_{h}^{-\frac12} )}\leq \Vert \partial_{3}\nabla^hw_{\varepsilon}^{h} (\cdot,0) \Vert _{L^{2} (\mathbb{R}^{+},\dot{H}_{h}^{-\frac12} )}\leq \Vert \partial_{3}\nabla^hw_{\varepsilon}^{h} \Vert _{L^{2} (\mathbb{R}^{+},L_{v}^{\infty}\dot{H}_{h}^{-\frac12} )}\,.
$$
The bounds on $w_{\varepsilon}$ given in Proposition~\ref{estimatesw} along with the Gagliardo-Nirenberg inequality lead to
$$
 \Vert \partial_{3}^{2}w_{\varepsilon}^{3} (\cdot,0) \Vert _{L^{2} (\mathbb{R}^{+},\dot{H}_{h}^{-\frac12} )} \begin{array}[t]{l}
\leq \Vert \partial_{3}\nabla^hw_{\varepsilon}^{h} \Vert _{L^{2} (\mathbb{R}^{+},L_{v}^{2}\dot{H}_{h}^{-\frac12} )}^\frac12 \Vert \partial_{3}^{2}\nabla^hw_{\varepsilon}^{h} \Vert _{L^{2} (\mathbb{R}^{+},L_{v}^{2}\dot{H}_{h}^{-\frac12} )}^\frac12\\
\lesssim\varepsilon^{2}\,.\end{array}
$$
Now  let us turn to the pressure term. Recall that
$$
-\Delta q_{\varepsilon} =\mbox{div} \, N_{\varepsilon}  \, ,
\quad \mbox{with} \quad  N_{\varepsilon}¬†  \eqdefa \overline v^h_{\varepsilon}\cdot\nabla^hw_{\varepsilon} = \mbox{div}_h \, (\overline v^h_{\varepsilon} \otimes w_{\varepsilon} ) 
$$
since~$\overline v^h_{\varepsilon}$ is divergence free. To estimate~$\partial_3  q_{\varepsilon}(\cdot,0)$ we   use Gagliardo-Nirenberg's inequality, according to which it suffices to estimate~$\partial_3  q_{\varepsilon}$ in~$L^2_v$ and in~$\dot H^1_v$.

Since~$(-\Delta) ^{-1} \mbox{div}_h \,  \mbox{div}$ is a zero order Fourier multiplier, we have
$$
  \Vert \partial_{3}q_{\varepsilon} \Vert _{L^{2} (\mathbb{R}^{+},H^{-\frac 1 2,1} )}\lesssim \Vert \partial_3 (\overline v^h_{\varepsilon} \otimes w_{\varepsilon} ) \Vert _{L^{2} (\mathbb{R}^{+},H^{-\frac 1 2,1} )}\,.
$$
On the one hand we write
$$
\begin{aligned}
\| w_{\varepsilon} \partial_3 \overline v^h_{\varepsilon}\|_{L^{2} (\mathbb{R}^{+},L_{v}^{2}\dot{H}_{h}^{-\frac12} )} & \lesssim
 \|  w_{\varepsilon} \|_{L^{2} (\mathbb{R}^{+},L_{v}^{2}\dot H_{h}^{\frac12 }) }
 \|  \partial_3 \overline v^h_{\varepsilon}\|_{L^{\infty} (\mathbb{R}^{+},L_{v}^{\infty}L_{h}^{2 }) }
  & \lesssim \eps^\frac12
\end{aligned}
$$
by Propositions~\ref{estimatesoverlinev} and~\ref{estimatesw}, and similarly
$$
\begin{aligned}
\|  \overline v^h_{\varepsilon}\partial_3 w_{\varepsilon}\|_{L^{2} (\mathbb{R}^{+},L_{v}^{2}\dot{H}_{h}^{-\frac12} )} & \lesssim
 \|  \partial_3 w_{\varepsilon} \|_{L^{2} (\mathbb{R}^{+},L_{v}^{2}\dot H_{h}^{\frac12 }) }
 \|  \overline v^h_{\varepsilon}\|_{L^{\infty} (\mathbb{R}^{+},L_{v}^{\infty}L_{h}^{2 }) }
  & \lesssim \eps^\frac12 \, .
\end{aligned}
$$
In the same way we find that
$$
\begin{aligned}
\| \partial_3(  \overline v^h_{\varepsilon} \otimes w_{\varepsilon})\|_{L^{2} (\mathbb{R}^{+},\dot H^{-\frac 1 2,1} )}   \lesssim  \eps^\frac32 \,  \end{aligned}.
$$
This ends  the proof of Proposition~\ref{estimateat0}.
     \end{proof}

\section{Estimates on the linear transport-diffusion equation}
 \label{appendix}

\setcounter{equation}{0}
In this appendix we shall prove Proposition~\ref{estimatesw}. It turns out to be convenient to rescale~$w_\eps$. Thus we define the
  vector field 
$$
W_{\varepsilon}(t,x)\eqdefa\big(\dfrac{w_{\varepsilon}^{h}}{\varepsilon},w_{\varepsilon}^{3} \big) (t,x_h,\varepsilon^{-1}x_3 )
$$ 
which satisfies
$$
 \left\{ \begin{array}{rl}
\partial_{t}W_{\varepsilon}+\overline V^h_{\varepsilon}\cdot\nabla^hW_{\varepsilon}-\Delta_{h}W_{\varepsilon}-\varepsilon^{2}\,\partial_{3}^{2}W_{\varepsilon} & =-\left(\nabla^hQ_{\varepsilon},\varepsilon^{2}\,\partial_{3}Q_{\varepsilon}\right)\\\mbox{div}\, W_{\varepsilon} & =0\\
W_{\varepsilon}\left(0,.\right) & =W_{\varepsilon ,0}\end{array}\right.
$$
where 
$$
 \overline V^h_{\varepsilon} (t,x)\eqdefa \overline v^h_{\varepsilon}(t,x_h,\varepsilon^{-1}x_3 ) \qquad\mbox{and}\qquad Q_{\varepsilon} (t,x)\eqdefa\varepsilon^{-1}q_{\varepsilon}  (t,x_h,\varepsilon^{-1}x_3 ) \, .
$$
Note that thanks to Proposition~\ref{estimatesoverlinev}, the vector field~$\partial^\alpha \overline V^h_{\varepsilon}$ is uniformly bounded in the space~$L^{\infty} (\mathbb{R}^{+},L_{v}^{2}\dot{H}_{h}^{s} )\cap L^{2} (\mathbb{R}^{+},L_{v}^{2}\dot{H}_{h}^{s+1} )$ for each~$\alpha \in {\mathbb N}^3$ and any~$s>-1$, and hence also in~$L^{\infty} (\mathbb{R}^{+},L_{v}^{\infty}\dot{H}_{h}^{s} )\cap L^{2} (\mathbb{R}^{+},L_{v}^{\infty}\dot{H}_{h}^{s+1} )$.

Similary we have defined
$$
W_{\varepsilon ,0}(x) \eqdefa \big(\dfrac{w_{\varepsilon,0}^{h}}{\varepsilon},w_{\varepsilon,0}^{3} \big) (x_h,\varepsilon^{-1}x_3 )
$$
and by construction it is bounded in~$\dot H^s(\R^3)$ for all~$ s \geq -1$.

Proposition~\ref{estimatesw} is a corollary of the next statement.
\begin{prop}
\label{propositionWeps}
{\sl
Under the assumptions of Theorem~{\rm\ref{spectral}}, the following results hold.
\begin{enumerate} 
\item\label{firstitempropositionWeps} For all~$s>-1$, and all $\alpha\in\mathbb{N}^{3}$,
 $\partial^{\alpha}W_{\varepsilon}$ is bounded in~$L^{\infty} (\mathbb{R}^{+},L_{v}^{2}\dot{H}_{h}^{s} )\cap L^{2} (\mathbb{R}^{+},L_{v}^{2}\dot{H}_{h}^{s+1} )$;  in particular  $\partial^{\alpha}W_{\varepsilon}$ is bounded in  $L^{\infty} (\mathbb{R}^{+},L_{v}^{\infty}\dot{H}_{h}^{s}  )\cap L^{2} (\mathbb{R}^{+},L_{v}^{\infty}\dot{H}_{h}^{s+1} )$.
 \item For all $\alpha\in\mathbb{N}^{3}$,
 $\partial^{\alpha}W_{\varepsilon}$ is bounded in  $L^{2} (\mathbb{R}^{+},L^{2} )$, hence in particular in $L^{2} (\mathbb{R}^{+},L_{v}^{\infty}L^2_h)$.
 \end{enumerate}
 }
 \end{prop}
  \begin{proof} Let us start by proving the first statement of the proposition. We notice that it is enough to   prove the result for $s\in ]-1,1 [$, and we shall argue by induction on $\alpha$.
 
$\bullet$ Let us start by considering the case~$\alpha = 0$. An energy estimate in~$L_{v}^{2}\dot{H}_{h}^{s}$ on the equation satisfied by~$W_{\varepsilon}$ gives
$$
\begin{array}[t]{l}
\dfrac{1}{2}\dfrac{\mbox{d}}{\mbox{d}t} \Vert W_{\varepsilon} \Vert _{L_{v}^{2}\dot{H}_{h}^{s}}^{2}+ \Vert \nabla^hW_{\varepsilon} \Vert _{L_{v}^{2}\dot{H}_{h}^{s}}^{2}+\varepsilon^{2} \Vert \partial_{3}W_{\varepsilon} \Vert _{L_{v}^{2}\dot{H}_{h}^{s}}^{2}\\[3mm]
\qquad\qquad\qquad\qquad=- \langle  \overline V^h_{\varepsilon}\cdot\nabla^hW_{\varepsilon},W_{\varepsilon} \rangle _{L_{v}^{2}\dot{H}_{h}^{s}}- \langle \nabla^hQ_{\varepsilon},W_{\varepsilon}^{h} \rangle _{L_{v}^{2}\dot{H}_{h}^{s}}- \langle \varepsilon^{2}\,\partial_{3}Q_{\varepsilon},W_{\varepsilon}^{3} \rangle _{L_{v}^{2}\dot{H}_{h}^{s}}\,.\end{array}
$$
For the non-linear term we have, by~\cite[Lemma 1.1]{chemin10} and for each given~$t$ and~$x_3$,
$$
\big | \langle \overline V^h_{\varepsilon}\cdot\nabla^hW_{\varepsilon},W_{\varepsilon} \rangle_{\dot{H}_{h}^{s}} (t,x_3) \big | \begin{array}[t]{l}
\lesssim \Vert \nabla^h \overline V^h_{\varepsilon}  (t,x_3)\Vert _{L_{h}^{2}} \Vert \nabla^hW_{\varepsilon}  (t,x_3)\Vert _{\dot{H}_{h}^{s}} \Vert W_{\varepsilon}  (t,x_3)\Vert _{\dot{H}_{h}^{s}}\\[2mm]
\leq\dfrac{1}{4} \Vert \nabla^hW_{\varepsilon} (t,x_3) \Vert _{\dot{H}_{h}^{s}}^{2}+C \Vert \nabla^h\overline V^h_{\varepsilon} (t,x_3) \Vert _{L_{h}^{2}}^{2} \Vert W_{\varepsilon}  (t,x_3)\Vert _{\dot{H}_{h}^{s}}^{2}\end{array}
$$
so after integration over~$x_3$, we find
$$
\begin{array}[t]{l}
\dfrac{1}{2}\dfrac{\mbox{d}}{\mbox{d}t}\Vert  W_{\varepsilon}\Vert   _{L_{v}^{2}\dot{H}_{h}^{s}}^{2}+\dfrac{3}{4}\Vert  \nabla^hW_{\varepsilon}\Vert   _{L_{v}^{2}\dot{H}_{h}^{s}}^{2}+\varepsilon^{2}\Vert  \partial_{3}W_{\varepsilon}\Vert   _{L_{v}^{2}\dot{H}_{h}^{s}}^{2}\\[3mm]
\qquad\qquad\qquad\leq C\Vert  \nabla^h\overline V^h_{\varepsilon}\Vert   _{L_{v}^{\infty}L_{h}^{2}}^{2}\Vert  W_{\varepsilon}\Vert   _{L_{v}^{2}\dot{H}_{h}^{s}}^{2}- \langle \nabla^hQ_{\varepsilon},W_{\varepsilon}^{h} \rangle _{L_{v}^{2}\dot{H}_{h}^{s}}- \langle \varepsilon^{2}\,\partial_{3}Q_{\varepsilon},W_{\varepsilon}^{3} \rangle _{L_{v}^{2}\dot{H}_{h}^{s}}\,.\end{array}
$$
Now let us study the pressure term.
As $W_{\varepsilon}$ is a divergence free vector field we have
$$
- \langle \nabla^hQ_{\varepsilon},W_{\varepsilon}^{h} \rangle _{L_{v}^{2}\dot{H}_{h}^{s}}- \langle \varepsilon^{2}\,\partial_{3}Q_{\varepsilon},W_{\varepsilon}^{3} \rangle _{L_{v}^{2}\dot{H}_{h}^{s}}= (\varepsilon^{2}-1 ) \langle \nabla^hQ_{\varepsilon},W_{\varepsilon}^{h} \rangle _{L_{v}^{2}\dot{H}_{h}^{s}}\,.
$$
We claim that 
\begin{equation}\label{claimQ}
 \big | \langle \nabla^hQ_{\varepsilon} (t) ,W_{\varepsilon}^{h}  (t) \rangle _{L_{v}^{2}\dot{H}_{h}^{s}} \big  |\leq\dfrac{1}{4} \Vert \nabla^hW_{\varepsilon} (t)  \Vert _{L_{v}^{2}\dot{H}_{h}^{s}}^{2}+C_{\varepsilon} (t)   \Vert W_{\varepsilon} (t) \Vert _{L_{v}^{2}\dot{H}_{h}^{s}}^{2}
\end{equation}
where $C_{\varepsilon}$  is uniformly bounded in $L^{1} (\mathbb{R}^{+} )$. Assuming that claim to be true, we infer (up to changing~$C_{\varepsilon}$) that 
$$
\dfrac{\mbox{d}}{\mbox{d}t} \Vert W_{\varepsilon} (t) \Vert _{L_{v}^{2}\dot{H}_{h}^{s}}^{2}+ \Vert \nabla^hW_{\varepsilon} (t) \Vert _{L_{v}^{2}\dot{H}_{h}^{s}}^{2}+\varepsilon^{2} \Vert \partial_{3}W_{\varepsilon} (t) \Vert _{L_{v}^{2}\dot{H}_{h}^{s}}^{2}\lesssim C_{\varepsilon} (t)  \Vert W_{\varepsilon} (t) \Vert _{L_{v}^{2}\dot{H}_{h}^{s}}^{2}\,.
$$
Thanks to Gronwall's lemma this gives
$$
 \Vert W_{\varepsilon} (t) \Vert _{L_{v}^{2}\dot{H}_{h}^{s}}^{2} +\int_{0}^{t} \Vert \nabla^hW_{\varepsilon} (t') \Vert _{L_{v}^{2}\dot{H}_{h}^{s}}^{2} \, d t'\lesssim \Vert W_{\varepsilon,0} \Vert _{L_{v}^{2}\dot{H}_{h}^{s}}^{2}\,,
$$
and the conclusion of Proposition~\ref{propositionWeps} (\ref{firstitempropositionWeps}), for~$\alpha = 0$ and~$-1<s<1$, comes from the a priori bounds on $W_{\varepsilon,0}$.
It remains to prove the claim~(\ref{claimQ}).
 For all real numbers~$r$, we have 
$$
 \big | \langle \nabla^hQ_{\varepsilon} (t) ,W_{\varepsilon}^{h}  (t) \rangle _{L_{v}^{2}\dot{H}_{h}^{s}} \big | \leq \Vert \nabla^hQ_{\varepsilon}(t)  \Vert _{L_{v}^{2}\dot{H}_{h}^{r}} \Vert W_{\varepsilon}^{h}(t)  \Vert _{L_{v}^{2}\dot{H}_{h}^{2s-r}} \, .
$$
As $W_{\varepsilon}$ is a  divergence free vector field we can write
$$
\mbox{div}\,  (\overline V^h_{\varepsilon}\cdot\nabla^hW_{\varepsilon} )=-\Delta_{h}Q_{\varepsilon}-\varepsilon^{2}\,\partial_{3}^{2}Q_{\varepsilon}\,.
$$
Then we define  
$$
M^h_{\varepsilon}\eqdefa \overline V^h_{\varepsilon}\cdot\nabla^hW_{\varepsilon}^{h}+\partial_{3} (W_{\varepsilon}^{3}\overline V^h_{\varepsilon} )
$$
and using the fact that~$\overline V^h_{\varepsilon}$ is divergence free, we have
$$
\mbox{div} (\overline V^h_{\varepsilon}\cdot\nabla^hW_{\varepsilon} )=\mbox{div}_{h} \, M^h_{\varepsilon} \,.
$$
It follows that
\begin{equation}\label{formulaQeps}
Q_{\varepsilon} = (-\Delta_{h} - \varepsilon^{2}\,\partial_{3}^{2})^{-1} \mbox{div}_{h} \, M^h_{\varepsilon}  \, ,  
\end{equation}
and since~$  \nabla^h(-\Delta_{h} - \varepsilon^{2}\,\partial_{3}^{2})^{-1} \mbox{div}_{h}$ is a zero-order Fourier multiplier, we infer that for all real numbers~$r$,
$$ \Vert \nabla^hQ_{\varepsilon} \Vert _{L_{v}^{2}\dot{H}_{h}^{r}}\leq \Vert M^h_{\varepsilon} \Vert _{L_{v}^{2}\dot{H}_{h}^{r}}\,,$$
and therefore
\begin{equation}\label{estiMW}
   | \langle \nabla^hQ_{\varepsilon} (t) ,W_{\varepsilon}^{h}  (t) \rangle _{L_{v}^{2}\dot{H}_{h}^{s}} |\leq \Vert M^h_\e(t) \Vert _{L_{v}^{2}\dot{H}_{h}^{r}} \Vert W_{\varepsilon}^{h} (t)\Vert _{L_{v}^{2}\dot{H}_{h}^{2s-r}}\,.
\end{equation}
We can estimate  $ \Vert M^h_\e \Vert _{L_{v}^{2}\dot{H}_{h}^{r}}$ as follows, thanks to the divergence-free condition on~$W_{\varepsilon}$:
$$
 \Vert M^h_\e \Vert _{L_{v}^{2}\dot{H}_{h}^{r}}\begin{array}[t]{l}
\leq \Vert \overline V^h_{\varepsilon}\cdot\nabla^hW_{\varepsilon} \Vert _{L_{v}^{2}\dot{H}_{h}^{r}}+ \Vert \partial_{3} (W_{\varepsilon}^{3}\overline V^h_{\varepsilon} ) \Vert _{L_{v}^{2}\dot{H}_{h}^{r}}\\[2mm]
\leq \Vert \overline V^h_{\varepsilon}\cdot\nabla^hW_{\varepsilon} \Vert _{L_{v}^{2}\dot{H}_{h}^{r}}+ \Vert W_{\varepsilon}^{3}\,\partial_{3}\overline V^h_{\varepsilon} \Vert _{L_{v}^{2}\dot{H}_{h}^{r}}+ \Vert \overline V^h_{\varepsilon}\,\mbox{div}_{h}W_{\varepsilon}^{h} \Vert _{L_{v}^{2}\dot{H}_{h}^{r}}\,.\end{array}
$$
Thanks to two-dimensional product laws, if $-1<r<0$ then we get
$$
 \Vert \overline V^h_{\varepsilon}\cdot\nabla^hW_{\varepsilon} \Vert _{L_{v}^{2}\dot{H}_{h}^{r}}+ \Vert \overline V^h_{\varepsilon}\,\mbox{div}_{h}W_{\varepsilon}^{h} \Vert _{L_{v}^{2}\dot{H}_{h}^{r}}\lesssim \Vert \overline V^h_{\varepsilon} \Vert _{L_{v}^{\infty}\dot{H}_{h}^{\frac12}} \Vert \nabla^hW_{\varepsilon} \Vert _{L_{v}^{2}\dot{H}_{h}^{r+ \frac12}}
$$
and
$$
 \Vert W_{\varepsilon}^{3}\partial_{3}\overline V^h_{\varepsilon} \Vert _{L_{v}^{2}\dot{H}_{h}^{r}}\lesssim \Vert \nabla \overline V^h_{\varepsilon} \Vert _{L_{v}^{\infty}L_{h}^{2}} \Vert W_{\varepsilon}^{3} \Vert _{L_{v}^{2}\dot{H}_{h}^{r+1}} \, .
$$
So  if $-1<r<0$,  then
\begin{equation}\label{estiM}
 \Vert M^h_\e \Vert _{L_{v}^{2}\dot{H}_{h}^{r}}\lesssim \Vert \overline V^h_{\varepsilon} \Vert _{L_{v}^{\infty}\dot{H}_{h}^{\frac12}} \Vert \nabla^hW_{\varepsilon} \Vert _{L_{v}^{2}\dot{H}_{h}^{r+ \frac12}}+ \Vert \nabla \overline V^h_{\varepsilon} \Vert _{L_{v}^{\infty}L_{h}^{2}} \Vert W_{\varepsilon}^{3} \Vert _{L_{v}^{2}\dot{H}_{h}^{r+1}}
 \end{equation}
and this   leads  to~(\ref{claimQ}) for $-1<s<1$, due to the following computations.

$\qquad \circ$ If $0<s<1$, we choose $r=s-1$ to get
$$
\|M^h_\e\|_{L_{v}^{2}\dot{H}_{h}^{s-1}} 
\lesssim  \Vert \overline V^h_{\varepsilon} \Vert _{L_{v}^{\infty}\dot{H}_{h}^{\frac12}} \Vert \nabla^hW_{\varepsilon}  \Vert _{L_{v}^{2}\dot{H}_{h}^{s-\frac12}}+ \Vert \nabla \overline V^h_{\varepsilon} \Vert _{L_{v}^{\infty}L_{h}^{2}} \Vert W_{\varepsilon}^{3} \Vert _{L_{v}^{2}\dot{H}_{h}^{s}} \, ,
$$
so by~(\ref{estiMW}) with~$r=s-1$,
we infer that  
\begin{equation}\label{estiQW}
\begin{aligned}
 \big | \langle \nabla^hQ_{\varepsilon}  ,W_{\varepsilon}^{h}   \rangle _{L_{v}^{2}\dot{H}_{h}^{s}} \big |&\lesssim \Vert \overline V^h_{\varepsilon} \Vert _{L_{v}^{\infty}\dot{H}_{h}^{\frac12}} \Vert \nabla^hW_{\varepsilon} \Vert _{L_{v}^{2}\dot{H}_{h}^{s-\frac12}} \Vert \nabla^hW_{\varepsilon}^{h} \Vert _{L_{v}^{2}\dot{H}_{h}^{s}} \\
 & \quad +\dfrac{1}{8} \Vert \nabla^h W_{\varepsilon}^{h} \Vert _{L_{v}^{2}\dot{H}_{h}^{s}}^{2}+C \Vert \nabla \overline V^h_{\varepsilon} \Vert _{L_{v}^{\infty}L_{h}^{2}}^{2} \Vert W_{\varepsilon}^{3} \Vert _{L_{v}^{2}\dot{H}_{h}^{s}}^{2}\,.
\end{aligned}
\end{equation}
We then use the interpolation inequality
$$
 \Vert \overline V^h_{\varepsilon} \Vert _{L_{v}^{\infty}\dot{H}_{h}^{\frac12}} \Vert \nabla^hW_{\varepsilon} \Vert _{L_{v}^{2}\dot{H}_{h}^{s-\frac12}}\lesssim  \Vert \overline V^h_{\varepsilon} \Vert _{L_{v}^{\infty}L_{h}^{2}}^{\frac12}\, \Vert \nabla^h \overline V^h_{\varepsilon} \Vert _{L_{v}^{\infty}L_{h}^{2}}^{\frac12}\, \Vert W_{\varepsilon} \Vert _{L_{v}^{2}\dot{H}_{h}^{s}}^{\frac12} \Vert \nabla^hW_{\varepsilon} \Vert _{L_{v}^{2}\dot{H}_{h}^{s}}^{\frac12}
 $$
along with the convexity inequality $ab\leq\tfrac{3}{4}a^{4/3}+\tfrac{1}{4}b^{4}$,  to get 
$$
\begin{aligned}
 \Vert \nabla^hW_{\varepsilon} \Vert _{L_{v}^{2}\dot{H}_{h}^{s}} \Vert \overline V^h_{\varepsilon} \Vert _{L_{v}^{\infty}\dot{H}_{h}^{\frac12}} \Vert \nabla^hW_{\varepsilon} \Vert _{L_{v}^{2}\dot{H}_{h}^{s-\frac12}}\leq\dfrac{1}{8} \Vert \nabla^hW_{\varepsilon} \Vert _{L_{v}^{2}\dot{H}_{h}^{s}}^{2}\\
 {}+C \Vert \overline V^h_{\varepsilon} \Vert _{L_{v}^{\infty}L_{h}^{2}}^{2} \Vert \nabla \overline V^h_{\varepsilon} \Vert _{L_{v}^{\infty}L_{h}^{2}}^{2} \Vert W_{\varepsilon} \Vert _{L_{v}^{2}\dot{H}_{h}^{s}}^{2} \, .
\end{aligned}
$$
It remains to define 
\begin{equation}\label{choiceCeps}
C_{\varepsilon} (t)\eqdefa C \Vert \nabla \overline V^h_{\varepsilon}(t)\Vert _{L_{v}^{\infty}L_{h}^{2}}^{2} \big(1+ \Vert \overline V^h_{\varepsilon}(t)\Vert _{L_{v}^{\infty}L_{h}^{2}}^{2} \big)
\end{equation}
to obtain from~(\ref{estiQW}) that
$$
 \big | \langle \nabla^hQ_{\varepsilon} (t) ,W_{\varepsilon}^{h}  (t) \rangle _{L_{v}^{2}\dot{H}_{h}^{s}} \big |\leq\dfrac{1}{4} \Vert \nabla^hW_{\varepsilon} (t)\Vert _{L_{v}^{2}\dot{H}_{h}^{s}}^{2}+C_{\varepsilon}(t) \Vert W_{\varepsilon} (t)\Vert _{L_{v}^{2}\dot{H}_{h}^{s}}^{2}\,.
$$
Notice that~$C_\eps$ belongs to~$L^1(\R^+)$ thanks to the uniform bounds on~$ \overline V^h_{\varepsilon}$ derived above from Proposition~\ref{estimatesoverlinev}.

$\qquad \circ$ If $s=0$, we choose $r=-\frac12 $ and hence by~(\ref{estiMW}) and~(\ref{estiM}),
$$
 \left | \langle \nabla^hQ_{\varepsilon}  ,W_{\varepsilon}^{h}  \rangle _{L^2}\right|\lesssim \Vert W_{\varepsilon}^{h} \Vert _{L_{v}^{2}\dot{H}_{h}^{\frac12}}\big( \Vert \overline V^h_{\varepsilon} \Vert _{L_{v}^{\infty}\dot{H}_{h}^{\frac12}} \Vert \nabla^hW_{\varepsilon} \Vert _{L^{2}}+ \Vert \nabla \overline V^h_{\varepsilon} \Vert _{L_{v}^{\infty}L_{h}^{2}} \Vert W_{\varepsilon}^{3} \Vert _{L_{v}^{2}\dot{H}_{h}^{\frac12}}\big) \, .
$$
By interpolation we infer that
$$
\longformule{
 \big | \langle \nabla^hQ_{\varepsilon}  ,W_{\varepsilon}^{h}  \rangle _{L^2} \big |\lesssim \Vert W_{\varepsilon}^{h} \Vert _{L^{2}}^{\frac12}\, \Vert \nabla^hW_{\varepsilon}^{h} \Vert _{L^{2}}^{3/2}\,  \Vert \overline V^h_{\varepsilon} \Vert _{L_{v}^{\infty}L_{h}^{2}}^{\frac12}\, \Vert \nabla \overline V^h_{\varepsilon} \Vert _{L_{v}^{\infty}L_{h}^{2}}^{\frac12} 
 }
 { {}
 + \Vert W_{\varepsilon}  \Vert _{L^{2}} \Vert \nabla^hW_{\varepsilon}  \Vert _{L^{2}} \Vert \nabla \overline V^h_{\varepsilon} \Vert _{L_{v}^{\infty}L_{h}^{2}}\,.
}
$$
The convexity inequality $ab\leq\tfrac{3}{4}a^{4/3}+\tfrac{1}{4}b^{4}$ implies that
\begin{equation}
\label{estimateW1}
\begin{split}
& \Vert W_{\varepsilon}^{h} \Vert _{L^{2}}^{\frac12}\, \Vert \nabla^hW_{\varepsilon}^{h} \Vert _{L^{2}}^{3/2}\, \Vert \overline V^h_{\varepsilon} \Vert _{L_{v}^{\infty}L_{h}^{2}}^{\frac12} \Vert \nabla \overline V^h_{\varepsilon} \Vert _{L_{v}^{\infty}L_{h}^{2}}^{\frac12}\\
&\qquad \qquad \qquad \qquad \qquad \qquad  {}   \leq\dfrac{1}{8} \Vert \nabla^hW_{\varepsilon} \Vert _{L^{2}}^{2} +C \Vert W_{\varepsilon}^{h} \Vert _{L^{2}}^{2} \Vert \overline V^h_{\varepsilon} \Vert _{L_{v}^{\infty}L_{h}^{2}}^{2} \Vert \nabla \overline V^h_{\varepsilon} \Vert _{L_{v}^{\infty}L_{h}^{2}}^{2}
 \end{split}
\end{equation}
and
\begin{equation}\label{estimateW2}
 \Vert W_{\varepsilon}  \Vert _{L^{2}} \Vert \nabla^hW_{\varepsilon} \Vert _{L^{2}} \Vert \nabla \overline V^h_{\varepsilon} \Vert _{L_{v}^{\infty}L_{h}^{2}}\leq\dfrac{1}{8} \Vert \nabla^hW_{\varepsilon} \Vert _{L^{2}}^{2}+C \Vert W_{\varepsilon} \Vert _{L^{2}}^{2}\Vert  \nabla \overline V^h_{\varepsilon}\Vert   _{L_{v}^{\infty}L_{h}^{2}}^{2} \, .
 \end{equation}
With the above choice~(\ref{choiceCeps}) for~$C_\eps$ we obtain 
$$
 \big | \langle \nabla^hQ_{\varepsilon} (t) ,W_{\varepsilon}^{h}  (t) \rangle _{L^2} \big |\leq\dfrac{1}{4}\Vert  \nabla^hW_{\varepsilon}(t)\Vert   _{L^2}^{2}+C_{\varepsilon}(t)\Vert  W_{\varepsilon}(t)\Vert   _{L^2}^{2} \, .
$$

$\qquad\circ$ Finally if $-1<s<0$, we   proceed slightly differently.
We recall that 
$$
\mbox{div}_{h} \, M^h_{\varepsilon} =-\Delta_{h}Q_{\varepsilon}-\varepsilon^{2}\,\partial_{3}^{2}Q_{\varepsilon}
$$ 
 and as $W_{\varepsilon}$ is   divergence free, we have 
$$
M^h_{\varepsilon}=\overline V^h_{\varepsilon}\cdot\nabla^hW_{\varepsilon}^{h}-\overline V^h_{\varepsilon}\,\mbox{div}_{h}W_{\varepsilon}^{h}+W_{\varepsilon}^{3}\,\partial_{3}\overline V^h_{\varepsilon} \, .
$$
Defining
$$
M^h_{\varepsilon,1} \eqdefa \mbox{div}_{h} \, (\overline V^h_{\varepsilon} \otimes W_{\varepsilon}^{h}- W_{\varepsilon}^{h} \otimes \overline V^h_{\varepsilon}) \quad \mbox{and} \quad M^h_{\varepsilon,2} \eqdefa W_{\varepsilon}\cdot \nabla \overline V^h_{\varepsilon} \, , 
$$
we can split $M^h_{\varepsilon} = M^h_{\varepsilon,1}+M^h_{\varepsilon,2}$ and estimate each term differently. 

Since~$\nabla^h (-\Delta_h - \eps^2 \partial_3^2)  \mbox{div}_{h} $ is a zero-order Fourier multiplier, 
$$
\big | \langle \nabla^hQ_{\varepsilon}  ,W_{\varepsilon}^{h}   \rangle _{L_{v}^{2}\dot{H}_{h}^{s}} \big |\leq\Vert  M^h_{\varepsilon,1}\Vert   _{L_{v}^{2}\dot{H}_{h}^{s-1}}\Vert  W_{\varepsilon}^{h}\Vert   _{L_{v}^{2}\dot{H}_{h}^{s+1}}+\Vert  M^h_{\varepsilon,2}\Vert   _{L_{v}^{2}\dot{H}_{h}^{s}}\Vert  W_{\varepsilon}^{h}\Vert   _{L_{v}^{2}\dot{H}_{h}^{s}} \, .
$$
Using two-dimensional  product laws we obtain
$$
\Vert  M^h_{\varepsilon,1}\Vert   _{L_{v}^{2}\dot{H}_{h}^{s-1}}\lesssim\Vert  \overline V^h_{\varepsilon}\, W_{\varepsilon}\Vert   _{L_{v}^{2}\dot{H}_{h}^{s}}\lesssim\Vert  \overline V^h_{\varepsilon}\Vert   _{L_{v}^{\infty}\dot{H}_{h}^{\frac12}}\Vert  W_{\varepsilon}\Vert   _{L_{v}^{2}\dot{H}_{h}^{s+ \frac12}}
$$
and
$$
\Vert  M^h_{\varepsilon,2}\Vert   _{L_{v}^{2}\dot{H}_{h}^{s}}\lesssim\Vert W_{\varepsilon} \cdot   \nabla \overline V^h_{\varepsilon}\Vert   _{L_{v}^{2}\dot{H}_{h}^{s}}\lesssim\Vert  \nabla \overline V^h_{\varepsilon}\Vert   _{L_{v}^{\infty}L^{2}}\Vert  W_{\varepsilon}\Vert   _{L_{v}^{2}\dot{H}_{h}^{s+1}} \, .
$$
Therefore, we get 
\begin{equation}\label{estis-1}
\begin{aligned}
| \langle \nabla^hQ_{\varepsilon}  ,W_{\varepsilon}^{h}  \rangle _{L_{v}^{2}\dot{H}_{h}^{s}} | \leq \Vert  \overline V^h_{\varepsilon}\Vert   _{L_{v}^{\infty}\dot{H}_{h}^{\frac12}}\Vert  W_{\varepsilon}\Vert   _{L_{v}^{2}\dot{H}_{h}^{s+ \frac12}}\Vert  \nabla^hW_{\varepsilon}^{h}\Vert   _{L_{v}^{2}\dot{H}_{h}^{s}}\\
 +\Vert  \nabla \overline V^h_{\varepsilon}\Vert   _{L_{v}^{\infty}L^{2}}\Vert  \nabla^hW_{\varepsilon}\Vert   _{L_{v}^{2}\dot{H}_{h}^{s}}\Vert  W_{\varepsilon}^{h}\Vert   _{L_{v}^{2}\dot{H}_{h}^{s}} \,.
\end{aligned}
\end{equation}
Then we use the interpolation inequality
$$
\Vert  W_{\varepsilon}\Vert   _{L_{v}^{2}\dot{H}_{h}^{s+ \frac12}}\Vert  \nabla^hW_{\varepsilon}^{h}\Vert   _{L_{v}^{2}\dot{H}_{h}^{s}}\lesssim\Vert  W_{\varepsilon}\Vert   _{L_{v}^{2}\dot{H}_{h}^{s}}^{\frac12}\Vert  \nabla^hW_{\varepsilon}^{h}\Vert   _{L_{v}^{2}\dot{H}_{h}^{s}}^{3/2}
$$
along with the convexity inequalities $ab\leq\tfrac{3}{4}a^{4/3}+\tfrac{1}{4}b^{4}$ and $ab\leq\tfrac{1}{2}a^{2}+\tfrac{1}{2}b^{2}$,  to infer that again with the choice~(\ref{choiceCeps}) for~$C_\eps$,
$$
\big | \langle \nabla^hQ_{\varepsilon} (t) ,W_{\varepsilon}^{h}  (t) \rangle _{L_{v}^{2}\dot{H}_{h}^{s}} \big |\leq\dfrac{1}{4}\Vert  \nabla^hW_{\varepsilon}(t)\Vert   _{L_{v}^{2}\dot{H}_{h}^{s}}^{2}+C_{\varepsilon}(t)\Vert  W_{\varepsilon}(t)\Vert   _{L_{v}^{2}\dot{H}_{h}^{s}}^{2}\,.
$$
The first result of the proposition is therefore proved in the case when~$\alpha = 0$ and~$-1<s<1$.

$\qquad \bullet$ To go further in the induction process, let~$ k \in \N$ be given and suppose the result proved for all $\alpha\in\mathbb{N} ^3$ such that $ | \alpha|\leq k$, still for~$-1<s<1$.
Now consider~$\alpha\in\mathbb{N} ^3$ such that~$ |\alpha |=k+1$. The vector field~$\partial^{\alpha}W_{\varepsilon}$ solves
$$
\partial_{t}\partial^{\alpha}W_{\varepsilon}+\partial^{\alpha} (\overline V^h_{\varepsilon}\cdot\nabla^hW_{\varepsilon} )-\Delta_{h}\partial^{\alpha}W_{\varepsilon}-\varepsilon^{2}\partial_{3}^{2}\partial^{\alpha}W_{\varepsilon}=- (\nabla^h\partial^{\alpha}Q_{\varepsilon},\varepsilon^{2}\partial_{3}\partial^{\alpha}Q_{\varepsilon} )\,.
$$
An energy estimate in~$L_{v}^{2}\dot{H}_{h}^{s}$ gives
$$
\begin{array}[t]{l}
\dfrac{1}{2}\dfrac{\mbox{d}}{\mbox{d}t}\Vert  \partial^{\alpha}W_{\varepsilon}\Vert   _{L_{v}^{2}\dot{H}_{h}^{s}}^{2}+ \langle \partial^{\alpha} (\overline V^h_{\varepsilon}\cdot\nabla^hW_{\varepsilon} ),\partial^{\alpha}W_{\varepsilon} \rangle _{L_{v}^{2}\dot{H}_{h}^{s}}+\Vert  \nabla^h\partial^{\alpha}W_{\varepsilon}\Vert   _{L_{v}^{2}\dot{H}_{h}^{s}}^{2}\\[3mm]
\qquad\qquad\qquad\qquad=- \langle \nabla^h\partial^{\alpha}Q_{\varepsilon},\partial^{\alpha}W_{\varepsilon}^{h} \rangle _{L_{v}^{2}\dot{H}_{h}^{s}}-\varepsilon^{2} \langle \partial_{3}\partial^{\alpha}Q_{\varepsilon},\partial^{\alpha}W_{\varepsilon}^{3} \rangle _{L_{v}^{2}\dot{H}_{h}^{s}}\,.\end{array}
$$
We split $ \langle \partial^{\alpha} (\overline V^h_{\varepsilon}\cdot\nabla^hW_{\varepsilon} ),\partial^{\alpha}W_{\varepsilon} \rangle _{L_{v}^{2}\dot{H}_{h}^{s}}$ into two contributions:
\begin{equation}\label{twocontribsalphabeta}
 \langle \overline V^h_{\varepsilon}\cdot\nabla^h\partial^{\alpha}W_{\varepsilon},\partial^{\alpha}W_{\varepsilon} \rangle _{L_{v}^{2}\dot{H}_{h}^{s}}+\underset{0 < \beta\leq \alpha}{\sum}C_{\beta} \langle \partial^{\beta}\overline V^h_{\varepsilon}\cdot\nabla^h\partial^{\alpha-\beta}W_{\varepsilon},\partial^{\alpha}W_{\varepsilon} \rangle _{L_{v}^{2}\dot{H}_{h}^{s}}\,.
\end{equation}
The first term in~(\ref{twocontribsalphabeta}) satisfies, as in~\cite[Lemma 1.1]{chemin10}
$$
| \langle \overline V^h_{\varepsilon}\cdot\nabla^h\partial^{\alpha}W_{\varepsilon},\partial^{\alpha}W_{\varepsilon} \rangle _{\dot{H}_{h}^{s}}| \begin{array}[t]{l}
\lesssim\Vert  \nabla^h\overline V^h_{\varepsilon}\Vert   _{L_{h}^{2}}\Vert  \nabla^h\partial^{\alpha}W_{\varepsilon}\Vert   _{\dot{H}_{h}^{s}}\Vert  \partial^{\alpha}W_{\varepsilon}\Vert   _{\dot{H}_{h}^{s}}\\[2mm]
\leq\dfrac{1}{4}\Vert  \nabla^h\partial^{\alpha}W_{\varepsilon}\Vert   _{\dot{H}_{h}^{s}}^{2}+C\Vert  \nabla^h\overline V^h_{\varepsilon}\Vert   _{L_{h}^{2}}^{2}\Vert  \partial^{\alpha}W_{\varepsilon}\Vert   _{\dot{H}_{h}^{s}}^{2}\end{array}
$$
so 
$$
| \langle \overline V^h_{\varepsilon}\cdot\nabla^h\partial^{\alpha}W_{\varepsilon},\partial^{\alpha}W_{\varepsilon} \rangle _{L_{v}^{2}\dot{H}_{h}^{s}} | \leq\dfrac{1}{4}\Vert  \nabla^h\partial^{\alpha}W_{\varepsilon}\Vert   _{L_{v}^{2}\dot{H}_{h}^{s}}^{2}+C\Vert  \nabla^h\overline V^h_{\varepsilon}\Vert   _{L_{v}^{\infty}L_{h}^{2}}^{2}\Vert  \partial^{\alpha}W_{\varepsilon}\Vert   _{L_{v}^{2}\dot{H}_{h}^{s}}^{2}\,.
$$
For the remaining terms in~(\ref{twocontribsalphabeta}), as  $\overline V^h_{\varepsilon}$ is a horizontal, divergence free vector field,   two-dimensional  product laws give
$$
|\langle \partial^{\beta}\overline V^h_{\varepsilon}\cdot\nabla^h\partial^{\alpha-\beta}W_{\varepsilon},\partial^{\alpha}W_{\varepsilon}\rangle _{\dot{H}_{h}^{s}} | \begin{array}[t]{l}
= \big |\langle \mbox{div}_{h} \big(\partial^{\beta}\overline V^h_{\varepsilon}\otimes\partial^{\alpha - \beta}W_{\varepsilon} \big),\partial^{\alpha}W_{\varepsilon} \rangle _{\dot{H}_{h}^{s}} \big |\\[3mm]
\lesssim\Vert  \partial^{\beta}\overline V^h_{\varepsilon}\otimes\partial^{\alpha-\beta}W_{\varepsilon}\Vert   _{\dot{H}_{h}^{s}}\Vert  \nabla^h\partial^{\alpha}W_{\varepsilon}\Vert   _{\dot{H}_{h}^{s}}\\[3mm]
\lesssim\Vert  \partial^{\beta}\overline V^h_{\varepsilon}\Vert   _{\dot{H}_{h}^{\frac{s+1}2 }}\Vert  \partial^{\alpha-\beta}W_{\varepsilon}\Vert   _{\dot{H}_{h}^{\frac{s+1}2}}
\Vert  \nabla^h\partial^{\alpha}W_{\varepsilon}\Vert   _{\dot{H}_{h}^{s}}\end{array}
$$
so
$$
|\langle \partial^{\beta}\overline V^h_{\varepsilon}\cdot\nabla^h\partial^{\alpha-\beta}W_{\varepsilon},\partial^{\alpha}W_{\varepsilon} \rangle _{L_{v}^{2}\dot{H}_{h}^{s}} | \leq\dfrac{1}{4}\Vert  \nabla^h\partial^{\alpha}W_{\varepsilon}\Vert   _{L_{v}^{2}\dot{H}_{h}^{s}}^{2}+C \Vert  \partial^{\beta}\overline V^h_{\varepsilon}\Vert   _{L_{v}^{\infty}\dot{H}_{h}^{\frac{s+1}2 }}^{2}\Vert  \partial^{\alpha-\beta}W_{\varepsilon}\Vert   _{L_{v}^{2}\dot{H}_{h}^{\frac{s+1}2 }}^{2}\,.
$$
Then we get
$$
\begin{aligned}
\dfrac{1}{2}\dfrac{\mbox{d}}{\mbox{d}t}\Vert  \partial^{\alpha}W_{\varepsilon}\Vert   _{L_{v}^{2}\dot{H}_{h}^{s}}^{2}+\dfrac{1}{2}\Vert  \nabla^h\partial^{\alpha}W_{\varepsilon}\Vert   _{L_{v}^{2}\dot{H}_{h}^{s}}^{2}
 +\varepsilon^{2}\Vert  \partial_{3}\partial^{\alpha}W_{\varepsilon}\Vert   _{L_{v}^{2}\dot{H}_{h}^{s}}^{2} \lesssim \|\nabla^h \overline V^h_{\varepsilon}\|_{L^\infty_v L^2_h}^2  \Vert  \partial^{\alpha}W_{\varepsilon}\Vert   _{L_{v}^{2}\dot{H}_{h}^{s}}^{2}\\
  +  \big | \langle \nabla^h\partial^{\alpha}Q_{\varepsilon},\partial^{\alpha}W_{\varepsilon}^{h} \rangle _{L_{v}^{2}\dot{H}_{h}^{s}} - \varepsilon^{2} \langle \partial_{3}\partial^{\alpha}Q_{\varepsilon},\partial^{\alpha}W_{\varepsilon}^{3} \rangle _{L_{v}^{2}\dot{H}_{h}^{s}} \big | \\
  +C\underset{0 < \beta\leq\alpha}{{\displaystyle \sum}}\Vert  \partial^{\beta}\overline V^h_{\varepsilon}\Vert   _{L_{v}^{\infty}\dot{H}_{h}^{\frac{s+1}2 }}^{2}\Vert  \partial^{\alpha-\beta}W_{\varepsilon}\Vert   _{L_{v}^{2}\dot{H}_{h}^{\frac{s+1}2 }}^{2}\,. 
\end{aligned}
$$
Now let us estimate the pressure term.
We recall that
$$
 - \langle \nabla^h\partial^{\alpha}Q_{\varepsilon},\partial^{\alpha}W_{\varepsilon}^{h} \rangle _{L_{v}^{2}\dot{H}_{h}^{s}}- \langle \varepsilon^{2}\,\partial_{3}\partial^{\alpha}Q_{\varepsilon},\partial^{\alpha}W_{\varepsilon}^{3} \rangle _{L_{v}^{2}\dot{H}_{h}^{s}}= (\varepsilon^{2}-1 ) \langle \nabla^h\partial^{\alpha}Q_{\varepsilon},\partial^{\alpha}W_{\varepsilon}^{h} \rangle _{L_{v}^{2}\dot{H}_{h}^{s}}  
 $$
and we claim that
\begin{equation}\label{anotherclaim}
 \big  | \langle \nabla^h\partial^{\alpha}Q_{\varepsilon},\partial^{\alpha}W_{\varepsilon}^{h}\rangle _{L_{v}^{2}\dot{H}_{h}^{s}} (t) \big |\leq\dfrac{1}{4}\Vert  \nabla^h\partial^{\alpha}W_{\varepsilon}(t) \Vert   _{L_{v}^{2}\dot{H}_{h}^{s}}^{2}+C_{1,\varepsilon}(t) +C_{2,\varepsilon}(t) \Vert  \partial^{\alpha}W_{\varepsilon}(t) \Vert   _{L_{v}^{2}\dot{H}_{h}^{s}}^{2}
\end{equation}
with $C_{1,\varepsilon}$ and $C_{2,\varepsilon}$ uniformly bounded in $L_{1} (\mathbb{R}^{+} )$.
By the induction assumption (noticing that~$(s+1)/2+\alpha-1 < \alpha$) we deduce that~$\displaystyle 
\underset{0 < \beta\leq\alpha}{\sum}\Vert  \partial^{\beta}\overline V^h_{\varepsilon}\Vert   _{L_{v}^{\infty}\dot{H}_{h}^{\frac{s+1}2}}^{2}\Vert  \partial^{\alpha-\beta}W_{\varepsilon}\Vert   _{L_{v}^{2}\dot{H}_{h}^{\frac{s+1}2}}^{2}
$
is uniformly bounded in~$L^{1} (\mathbb{R}^{+} )$ so  up to changing~$C_{1,\varepsilon}$ and~$C_{2,\varepsilon}$ we get 
$$
\dfrac{\mbox{d}}{\mbox{d}t}\Vert  \partial^{\alpha}W_{\varepsilon} (t) \Vert   _{L_{v}^{2}\dot{H}_{h}^{s}}^{2}+\Vert  \nabla^h\partial^{\alpha}W_{\varepsilon} (t) \Vert   _{L_{v}^{2}\dot{H}_{h}^{s}}^{2}\leq C_{1,\varepsilon} (t) +C_{2,\varepsilon} (t) \Vert  \partial^{\alpha}W_{\varepsilon} (t) \Vert   _{L_{v}^{2}\dot{H}_{h}^{s}}^{2} \,.
$$
Using Gronwall's lemma in turn this implies that
$$
\Vert  \partial^{\alpha}W_{\varepsilon} (t )\Vert   _{L_{v}^{2}\dot{H}_{h}^{s}}^{2}+\int_{0}^{t}\Vert  \nabla^h\partial^{\alpha}W_{\varepsilon} (t' )\Vert   _{L_{v}^{2}\dot{H}_{h}^{s}}^{2}\mbox{d}t'\lesssim\Vert  \partial^{\alpha}W_{\varepsilon,0}\Vert   _{L_{v}^{2}\dot{H}_{h}^{s}}^{2}
$$
and the bounds on $W_{\varepsilon ,0}$ conclude  the proof if~$-1<s<1$.
 It remains to prove the estimate~(\ref{anotherclaim}) on the pressure term.
We shall adapt the  computations of the case~$\alpha=0$. We define
$$
N_{\varepsilon,\alpha,\beta}\eqdefa\partial^{\beta}\overline V^h_{\varepsilon}\cdot\nabla^h\partial^{\alpha-\beta}W_{\varepsilon}^{h}+\partial_{3} (\partial^{\alpha-\beta}W_{\varepsilon}^{3}\partial^{\beta}\overline V^h_{\varepsilon} )
$$
and recalling~(\ref{formulaQeps}) we get, since~$\nabla^h (-\Delta_h - \eps^2 \partial_3^2)^{-1} \mbox{div}_h$ is a Fourier multiplier of order 0,
$$
 \langle \nabla^h\partial^{\alpha}Q_{\varepsilon},\partial^{\alpha}W_{\varepsilon}^{h} \rangle _{L_{v}^{2}\dot{H}_{h}^{s}}\lesssim\underset{0 \leq \beta\leq\alpha}{\sum}\Vert  N_{\varepsilon,\alpha,\beta}\Vert   _{L_{v}^{2}\dot{H}_{h}^{r_{\beta}}}\Vert  \partial^{\alpha}W_{\varepsilon}^{h} (t,\cdot)\Vert   _{L_{v}^{2}\dot{H}_{h}^{2s-r_{\beta}}}
$$
where~$r_\beta$ is any real number.
Then we define
$$
\left(*\right)_{\alpha,\beta}:=\Vert  N_{\varepsilon,\alpha,\beta}\Vert   _{L_{v}^{2}\dot{H}_{h}^{r_{\beta}}}\Vert  \partial^{\alpha}W_{\varepsilon}^{h} (t,\cdot)\Vert   _{L_{v}^{2}\dot{H}_{h}^{2s-r_{\beta}}} \, .
$$
The term $ (* )_{\alpha,0}$ can be treated as    was done for~$\alpha=0$, changing~$W_{\varepsilon}^{h}$ into~$\partial^{\alpha}W_{\varepsilon}^{h}$. 
So we have, as in the proof of~(\ref{claimQ}),
\begin{equation}\label{estialpha0}
\left| (* )_{\alpha,0}\right|\leq\dfrac{1}{8}\Vert  \nabla^h\partial^{\alpha}W_{\varepsilon}\Vert   _{L_{v}^{2}\dot{H}_{h}^{s}}^{2}+C\Vert  \partial^{\alpha}W_{\varepsilon}\Vert   _{L_{v}^{2}\dot{H}_{h}^{s}}^{2}\Vert  \nabla^h\partial^{\alpha}\overline V^h_{\varepsilon}\Vert   _{L_{v}^{\infty}L_{h}^{2}}^{2}\big(1+\Vert  \partial^{\alpha}\overline V^h_{\varepsilon}\Vert   _{L_{v}^{\infty}L_{h}^{2}}^{2} \big) \, .
\end{equation}
For the others terms  we have the following estimates.

$\qquad \circ$ If $0<s<1$ we choose~$r_\beta = s-1$ like in the case~$\alpha = 0$, and  as in~(\ref{estiQW}) we obtain
$$
\begin{aligned}
\sum_{0<\beta\leq\alpha}\left| (* )_{\alpha,\beta}\right|\leq \dfrac{1}{8}\Vert  \nabla^h\partial^{\alpha}W_{\varepsilon}\Vert   _{L_{v}^{2}\dot{H}^{s}_h}^{2}& +C\,\underset{0<\beta\leq\alpha}{{\displaystyle \sum}}\Vert  \partial^{\beta}\overline V^h_{\varepsilon}\Vert   _{L_{v}^{\infty}\dot{H}_{h}^{\frac12}}^{2}\Vert  \nabla^h\partial^{\alpha-\beta}W_{\varepsilon}\Vert   _{L_{v}^{2}\dot{H}_{h}^{s-\frac12}}^{2}\\
& +{} C \underset{0<\beta\leq\alpha}{{\displaystyle \sum}}\Vert  \nabla\partial^{\beta}\overline V^h_{\varepsilon}\Vert   _{L_{v}^{\infty}L_{h}^{2}}^{2}\Vert  \partial^{\alpha-\beta}W_{\varepsilon}^{3}\Vert   _{L_{v}^{2}\dot{H}_{h}^{s}}^{2}\,. \end{aligned}
$$
Then we define, recalling~(\ref{estialpha0}),
$$
C_{1,\varepsilon} \eqdefa C\,\underset{0<\beta\leq\alpha}{{\displaystyle \sum}}\Vert  \partial^{\beta}\overline V^h_{\varepsilon}\Vert   _{L_{v}^{\infty}\dot{H}_{h}^{\frac12}}^{2}\Vert  \nabla^h\partial^{\alpha-\beta}W_{\varepsilon}\Vert   _{L_{v}^{2}\dot{H}_{h}^{s-\frac12}}^{2}+C\,\underset{0<\beta\leq\alpha}{{\displaystyle \sum}}\Vert  \nabla\partial^{\beta}\overline V^h_{\varepsilon}\Vert   _{L_{v}^{\infty}L_{h}^{2}}^{2}\Vert  \partial^{\alpha-\beta}W_{\varepsilon}^{3}\Vert   _{L_{v}^{2}\dot{H}_{h}^{s}}^{2}
$$
and
$$
C_{2,\varepsilon} \eqdefa  C\Vert  \nabla^h\partial^{\alpha}\overline V^h_{\varepsilon}\Vert   _{L_{v}^{\infty}L_{h}^{2} }^{2}\big(1+\Vert  \partial^{\alpha}\overline V^h_{\varepsilon}\Vert   _{L_{v}^{\infty}L_{h}^{2}  }^{2}\big)
$$
to get
$$
 \sum_{0 \leq \beta \leq \alpha} \big|(* )_{\alpha,\beta} \big|\leq\dfrac{1}{4}\Vert  \nabla^h\partial^{\alpha}W_{\varepsilon}\Vert   _{L_{v}^{2}\dot{H}_{h}^{s}}^{2}+C_{1,\varepsilon}+C_{2,\varepsilon}\Vert  \partial^{\alpha}W_{\varepsilon}\Vert   _{L_{v}^{2}\dot{H}_{h}^{s}}^{2}\,.
$$
Note that the famillies  $ (C_{1,\varepsilon} )_{\varepsilon>0}$ and $ (C_{2,\varepsilon} )_{\varepsilon>0}$ are bounded in $L^{1} (\mathbb{R}^{+} )$ thanks to the induction assumption and Proposition~\ref{estimatesoverlinev}.

$\qquad\circ$ If $s=0$ then following the steps leading to~(\ref{estimateW1})-(\ref{estimateW2}) we choose~$r_\beta  =-1/2$ and write
$$
 \big| (* )_{\alpha,\beta} \big |
 \lesssim \Vert  \partial^{\alpha}W_{\varepsilon}^{h}\Vert   _{L_{v}^{2}\dot{H}_{h}^{\frac12}} \big( \Vert  \partial^{\beta}\overline V^h_{\varepsilon}\Vert   _{L_{v}^{\infty}\dot{H}_{h}^{\frac12}}\Vert  \nabla^h\partial^{\alpha-\beta}W_{\varepsilon}\Vert   _{L^{2}} + \Vert  \nabla\partial^{\beta}\overline V^h_{\varepsilon}\Vert   _{L_{v}^{\infty}L_{h}^{2}}\Vert  \partial^{\alpha-\beta}W_{\varepsilon}^{3}\Vert   _{L_{v}^{2}\dot{H}_{h}^{\frac12}}\big)
$$
so, by interpolation, we get
$$
\big | (* )_{\alpha,\beta} \big |\begin{array}[t]{l}
\lesssim \Vert  \partial^{\alpha}W_{\varepsilon}^{h}\Vert   _{L^{2}}^{\frac12} \, \Vert  \partial^{\alpha}\nabla^hW_{\varepsilon}^{h}\Vert   _{L^{2}}^{\frac12}\,\Vert  \partial^{\beta}\overline V^h_{\varepsilon}\Vert   _{L_{v}^{\infty}\dot{H}_{h}^{\frac12}}\Vert  \nabla^h\partial^{\alpha-\beta}W_{\varepsilon}\Vert   _{L^{2}}\\[2mm]
\quad+ \Vert  \partial^{\alpha}W_{\varepsilon}^{h}\Vert   _{L^{2}}^{\frac12} \, \Vert  \partial^{\alpha}\nabla^hW_{\varepsilon}^{h}\Vert   _{L^{2}}^{\frac12}\,\Vert  \nabla\partial^{\beta}\overline V^h_{\varepsilon}\Vert   _{L_{v}^{\infty}L_{h}^{2}}\Vert  \partial^{\alpha-\beta}W_{\varepsilon}^{3}\Vert   _{L_{v}^{2}\dot{H}_{h}^{\frac12}} \, .
\end{array}
$$
When~$\beta>0$, the convexity inequality  $abc\leq\tfrac{1}{4}a^{4}+\tfrac{1}{4}b^{4}+\tfrac{1}{2}c^{2}$ leads to
$$
\underset{0<\beta\leq \alpha}{{\displaystyle \sum}} \big | (* )_{\alpha,\beta} \big |\leq\begin{array}[t]{l}
\dfrac{1}{8}\Vert  \nabla^h\partial^{\alpha}W_{\varepsilon}\Vert   _{L^{2}}^{2}+C\,\underset{0<\beta\leq \alpha}{{\displaystyle \sum}}\Vert  \partial^{\alpha}W_{\varepsilon}^{h}\Vert   _{L^{2}}^{2}\big(\Vert  \partial^{\beta}\overline V^h_{\varepsilon}\Vert   _{L_{v}^{\infty}\dot{H}_{h}^{\frac12}}^{4}+\Vert  \nabla\partial^{\beta}\overline V^h_{\varepsilon}\Vert   _{L_{v}^{\infty}L_{h}^{2}}^{4}\big)\\[3mm]
{}+C\,\underset{0<\beta\leq \alpha}{{\displaystyle \sum}}\big(\Vert  \nabla^h\partial^{\alpha-\beta}W_{\varepsilon}\Vert   _{L^{2}}^{2}+\Vert  \nabla^h\partial^{\alpha-\beta}W_{\varepsilon}^{3}\Vert   _{L_{v}^{2}\dot{H}_{h}^{-\frac12}}^{2}\big)\,.\end{array}
$$
We define 
$$
C_{1,\varepsilon} \eqdefa C\,\underset{0<\beta\leq \alpha}{\sum}\big(\Vert  \nabla^h\partial^{\alpha-\beta}W_{\varepsilon}\Vert   _{L^{2}}^{2}+\Vert  \nabla^h\partial^{\alpha-\beta}W_{\varepsilon}^{3}\Vert   _{L_{v}^{2}\dot{H}_{h}^{-\frac12}}^{2} \big)
$$
and
$$
C_{2,\varepsilon} \eqdefa C\Vert  \nabla^h\partial^{\alpha}\overline V^h_{\varepsilon}\Vert   _{L_{v}^{\infty}L_{h}^{2}}^{2} \big(1+\Vert  \partial^{\alpha}\overline V^h_{\varepsilon}\Vert   _{L_{v}^{\infty}L_{h}^{2}}^{2} \big)+C\,\underset{0<\beta\leq \alpha}{{\displaystyle \sum}} \big(\Vert  \partial^{\beta}\overline V^h_{\varepsilon}\Vert   _{L_{v}^{\infty}\dot{H}_{h}^{\frac12}}^{4}+\Vert  \nabla\partial^{\beta}\overline V^h_{\varepsilon}\Vert   _{L_{v}^{\infty}L_{h}^{2}}^{4} \big)
$$ 
to get when~$s=0$ and recalling~(\ref{estialpha0}),
$$
 \sum_{0< \beta \leq \alpha} \big|(* )_{\alpha,\beta} \big|\leq\dfrac{1}{4}\Vert  \nabla^h\partial^{\alpha}W_{\varepsilon}\Vert   _{L^2}^{2}+C_{1,\varepsilon}+C_{2,\varepsilon}\Vert  \partial^{\alpha}W_{\varepsilon}\Vert   _{L^2}^{2}\,.
$$
Again note that the famillies  $ (C_{1,\varepsilon} )_{\varepsilon>0}$ and $ (C_{2,\varepsilon} )_{\varepsilon>0}$ are bounded in $L^{1} (\mathbb{R}^{+} )$ thanks to the induction assumption and Proposition~\ref{estimatesoverlinev}.

$\qquad\circ$ If $-1<s<0$ then following the computations leading to~(\ref{estis-1}), we write
$$
\begin{aligned}
\left| (* )_{\alpha,\beta}\right|\lesssim\Vert  \partial^{\beta}\overline V^h_{\varepsilon}\Vert   _{L_{v}^{\infty}\dot{H}_{h}^{\frac12}}\Vert  \partial^{\alpha-\beta}W_{\varepsilon}\Vert   _{L_{v}^{2}\dot{H}_{h}^{s+ \frac12}}\Vert  \nabla\partial^{\alpha}W_{\varepsilon}^{h}\Vert   _{L_{v}^{2}\dot{H}_{h}^{s}}\\
+\Vert  \nabla\partial^{\beta}\overline V^h_{\varepsilon}\Vert   _{L_{v}^{\infty}L^{2}}\Vert  \nabla\partial^{\alpha-\beta}W_{\varepsilon}\Vert   _{L_{v}^{2}\dot{H}_{h}^{s}}\Vert  \partial^{\alpha}W_{\varepsilon}^{h}\Vert   _{L_{v}^{2}\dot{H}_{h}^{s}}
\end{aligned}
$$
so
$$
\underset{0<\beta\leq \alpha}{\sum}\left| (* )_{\alpha,\beta}\right|\leq\begin{array}[t]{l}
\dfrac{1}{4}\Vert  \nabla^h\partial^{\alpha}W_{\varepsilon}\Vert   _{L_{v}^{2}\dot{H}_{h}^{s}}^{2}+\underset{\beta<\alpha}{{\displaystyle \sum}}\Vert  \partial^{\beta}\overline V^h_{\varepsilon}\Vert   _{L_{v}^{\infty}\dot{H}_{h}^{\frac12}}^{2}\Vert  \partial^{\alpha-\beta}W_{\varepsilon}\Vert   _{L_{v}^{2}\dot{H}_{h}^{s+ \frac12}}^{2}\\[4mm]
+\underset{\beta<\alpha}{{\displaystyle \sum}}\Vert  \nabla\partial^{\beta}\overline V^h_{\varepsilon}\Vert   _{L_{v}^{\infty}L^{2}}\Vert  \nabla\partial^{\alpha-\beta}W_{\varepsilon}\Vert   _{L_{v}^{2}\dot{H}_{h}^{s}}\Vert  \partial^{\alpha}W_{\varepsilon}^{h}\Vert   _{L_{v}^{2}\dot{H}_{h}^{s}}\,.\end{array}
$$
In this case, we define
$$
C_{1,\varepsilon} \eqdefa C\,\underset{0<\beta\leq\alpha}{\sum}\Vert  \partial^{\beta}\overline V^h_{\varepsilon}\Vert   _{L_{v}^{\infty}\dot{H}_{h}^{\frac12}}^{2}\Vert  \partial^{\alpha-\beta}W_{\varepsilon}\Vert   _{L_{v}^{2}\dot{H}_{h}^{s+ \frac12}}^{2}
$$
and
$$
C_{2,\varepsilon} \eqdefa C\Vert  \nabla^h\partial^{\alpha}\overline V^h_{\varepsilon}\Vert   _{L_{v}^{\infty}L_{h}^{2}}^{2} \big(1+\Vert  \partial^{\alpha}\overline V^h_{\varepsilon}\Vert   _{L_{v}^{\infty}L_{h}^{2}}^{2} \big)+C\,\underset{0<\beta\leq\alpha}{{\displaystyle \sum}}\Vert  \nabla\partial^{\beta}\overline V^h_{\varepsilon}\Vert   _{L_{v}^{\infty}L^{2}}\Vert  \nabla\partial^{\alpha-\beta}W_{\varepsilon}\Vert   _{L_{v}^{2}\dot{H}_{h}^{s}}
$$
which as before are  bounded in $L^{1} (\mathbb{R}^{+} )$,
and we obtain, recalling~(\ref{estialpha0}),
$$
\sum_{0 \leq \beta \leq \alpha } (* )_{\alpha,\beta}\leq\dfrac{1}{4}\Vert  \nabla^h\partial^{\alpha}W_{\varepsilon}\Vert   _{L_{v}^{2}\dot{H}_{h}^{s}}^{2}+C_{1,\varepsilon}+C_{2,\varepsilon}\Vert  \partial^{\alpha}W_{\varepsilon}\Vert   _{L_{v}^{2}\dot{H}_{h}^{s}}^{2}\,.
$$
The first part of the proposition is proved.

\medskip

Now let us turn to the second part. As noted above,
for all $\alpha\in\mathbb{N}^{3}$, $\partial^{\alpha}W_{\varepsilon}$ satisfies
$$ \partial_{t}\partial^{\alpha}W_{\varepsilon}+\partial^{\alpha} (\overline V^h_{\varepsilon}\cdot\nabla^hW_{\varepsilon} )-\Delta_{h}\partial^{\alpha}W_{\varepsilon}-\varepsilon^{2}\,\partial_{3}^{2}\partial^{\alpha}W_{\varepsilon}=-\partial^{\alpha} (\nabla^hQ_{\varepsilon},\varepsilon^{2}\,\partial_{3}Q_{\varepsilon} )\,.
$$
Defining
$$
 g_{\varepsilon} \eqdefa  \overline V^h_{\varepsilon}\cdot\nabla^hW_{\varepsilon} + (\nabla^hQ_{\varepsilon},\varepsilon^{2}\,\partial_{3}Q_{\varepsilon} ) \, ,
$$
an energy estimate in~$L^2_v\dot H^{-1}_h$ gives
\begin{equation}\label{estimatedotH-1}
\begin{aligned}
\dfrac{1}{2} \Vert \partial^{\alpha}W_{\varepsilon} (t)\Vert _{L^2_v \dot H_{h}^{-1}}^{2} +{\displaystyle \int_{0}^{t}\Vert \partial^{\alpha}W_{\varepsilon} (t'  )\Vert _{L^{2}}^{2}\mbox{d}t'}\leq\dfrac{1}{2}\Vert \partial^{\alpha}W_{\varepsilon,0}\Vert _{L^2_v\dot H_{h}^{-1}}^{2} \\+{\displaystyle \int_{0}^{t} \big | \left\langle \partial^\alpha g_{\varepsilon},\partial^{\alpha}W_{\varepsilon}\right\rangle _{L^2_v\dot H_{h}^{-1}} (t')\big |\, d t'} \, .
\end{aligned}
\end{equation}
We define $K_{\varepsilon} (t )\eqdefa\displaystyle \sup_{0\leq t'\leq t}   \Vert \partial^{\alpha}W_{\varepsilon}(t') \Vert _{\dot H_{h}^{-1}}$, so that
$$
\dfrac{1}{2}K_{\varepsilon}^{2} (t ) \leq\dfrac{1}{2}\Vert \partial^{\alpha}W_{\varepsilon,0}\Vert _{L^2_v \dot H_{h}^{-1}}^{2} +K_{\varepsilon}(t){\displaystyle \int_{0}^{t}\Vert \partial^\alpha g_{\varepsilon}(t')\Vert _{L^2_v \dot H_{h}^{-1}}\, d t'} \, .
$$
This implies that
\begin{equation}\label{estiKeps}
\dfrac{1}{4}K_{\varepsilon}^{2} (t)\leq\dfrac{1}{2}\Vert \partial^{\alpha}W_{\varepsilon,0}\Vert _{L^2_v \dot H_{h}^{-1}}^{2} +\Vert \partial^\alpha g_{\varepsilon}\Vert _{L^{1}  (\mathbb{R}^{+},L^2_v \dot H_{h}^{-1} )}^{2} \,.
\end{equation}
But according to~(\ref{estimatedotH-1}) we know that
$$
\displaystyle \int_{0}^{t}\Vert \partial^{\alpha}W_{\varepsilon} (t')\Vert _{L^{2}}^{2}\, d t' \leq \dfrac{1}{2}\Vert \partial^{\alpha}W_{\varepsilon,0}\Vert _{L^2_v \dot H_{h}^{-1}}^{2} +K_{\varepsilon} (t){\displaystyle \int_{0}^{t}\Vert \partial^\alpha g_{\varepsilon} (t')\Vert _{L^2_v \dot H_{h}^{-1}}\, d t'} \, ,
$$
so with~(\ref{estiKeps}) we infer that
$$
\int_{0}^{t}\Vert \partial^{\alpha}W_{\varepsilon} (t')\Vert _{L^{2}}^{2} dt'\lesssim\Vert \partial^{\alpha}W_{\varepsilon,0}\Vert _{L^2_v\dot H_{h}^{-1}}^{2}+\Vert \partial^\alpha g_{\varepsilon}\Vert _{L^{1} (\mathbb{R}^{+},L^2_v \dot H_{h}^{-1}  )}^{2}  \, .
$$
It remains to estimate~$\Vert \partial^\alpha g_{\varepsilon}\Vert_{L^{1} (\mathbb{R}^{+},L^2_v \dot H_{h}^{-1}  )}$. 
As $\overline V^h_{\varepsilon}$ is a divergence free vector field, we have  
\begin{equation}\label{alphabeta}
\begin{aligned}
\Vert \partial^{\alpha} (\overline V^h_{\varepsilon}\cdot\nabla^hW_{\varepsilon} ) \Vert _{L^{1} (\mathbb{R}^{+},L^2_v \dot{H}_{h}^{-1})}& \leq\Vert \partial^{\alpha} (\overline V^h_{\varepsilon}\otimes W_{\varepsilon} ) \Vert _{L^{1} (\mathbb{R}^{+},L ^{2})} \\
&\lesssim\underset{0\leq\beta\leq\alpha}{\sum} \Vert \partial^{\beta}\overline V^h_{\varepsilon} \Vert _{L^2(\R^+;L^2_v\dot{H}_{h}^{1/2})}   \Vert \partial^{\alpha-\beta}W_{\varepsilon} \Vert _{L^2(\R^+;L^\infty_v\dot{H}_{h}^{1/2})}
\end{aligned}
\end{equation}
 which gives the expected bound due to Proposition~\ref{propositionWeps} (\ref{firstitempropositionWeps}) proved above. 
On the other hand, we recall that as computed in~(\ref{formulaQeps}),
$$
\Delta_h Q_{\varepsilon}-\varepsilon^{2}\,\partial_{3}^2Q_{\varepsilon} =\mbox{div}_{h} \big(\overline V^h_{\varepsilon}\cdot\nabla^{h}W_{\varepsilon}^{h}+\partial_{3} (W_{\varepsilon}^{3}\, \overline V^h_{\varepsilon} ) \big) \,.
$$
so since~$(\Delta_h -\varepsilon^{2}\,\partial_{3}^2)^{-1} \nabla_h\mbox{div}_{h}$   and~$(\Delta_h -\varepsilon^{2}\,\partial_{3}^2)^{-1} \eps \partial_3 \mbox{div}_{h}$ are zero-order Fourier multipliers, the same estimates give the expected a priori bound on~$(\nabla^hQ_{\varepsilon},\varepsilon^{2}\,\partial_{3}Q_{\varepsilon} ),$ and the result follows.
  \end{proof}

\end{document}